\documentclass[12pt]{article}
\usepackage{amssymb}
\usepackage{url}
\usepackage[dvips]{graphicx}
\usepackage{graphics}
\usepackage[dvips]{color}
\usepackage{xcolor}
\usepackage[figuresright]{rotating}
\usepackage{multirow}
\usepackage[margin=1in]{geometry}
\usepackage{algorithm}
\usepackage{algorithmicx}
\usepackage{algpseudocode}
\usepackage{amsmath,dsfont}
\usepackage{amsthm}
\usepackage{amsmath}
\usepackage{amsfonts}
\usepackage{epic,eepic}
\usepackage{epsfig}
\usepackage{booktabs} 
\usepackage[toc,page,title,titletoc,header]{appendix}
\usepackage[font={small}]{caption}
\usepackage{longtable}
\usepackage{pdflscape}

\newtheorem{thm}{Theorem}[section]
\newtheorem{lem}[thm]{Lemma}

\def\R{{\mathbb R}}
\newcommand\norm[1]{\left\lVert#1\right\rVert}

\DeclareMathOperator*{\argmin}{argmin}

\usepackage{amsfonts}

\date{ }
\author{Rujun Jiang\thanks{School of Data Science, Fudan University, Shanghai, China, rjjiang@fudan.edu.cn}
 \and Man-Chung Yue\thanks{Department of Applied Mathematics, The Hong Kong Polytechnic University, Hung Hom, Hong Kong, manchung.yue@polyu.edu.hk}\and Zhishuo Zhou\thanks{School of Data Science, Fudan University, Shanghai, China, zhouzs18@fudan.edu.cn}}

\title{An accelerated first-order method with complexity analysis for solving cubic regularization subproblems}
\usepackage{varioref}

\begin{document}

\maketitle
\begin{abstract}
We propose a first-order method to solve the cubic regularization subproblem (CRS) based on a novel reformulation. The reformulation is a constrained convex optimization problem whose feasible region admits an easily computable projection. Our reformulation requires computing the minimum eigenvalue of the Hessian. To avoid the expensive computation of the exact minimum eigenvalue, we develop a surrogate problem to the reformulation where the exact minimum eigenvalue is replaced with an approximate one.
We then apply first-order methods such as the  Nesterov's  accelerated projected gradient method  (APG) and projected Barzilai-Borwein method to solve the surrogate problem. As our main theoretical contribution, we show that when an $\epsilon$-approximate minimum eigenvalue is computed by the Lanczos method and the surrogate problem is approximately solved by APG, our approach returns an $\epsilon$-approximate solution to CRS in $\tilde O(\epsilon^{-1/2})$ matrix-vector multiplications (where $\tilde O(\cdot)$ hides the logarithmic factors).
Numerical experiments show that our methods are comparable to and outperform the Krylov subspace method in the easy and hard cases, respectively. 
We further implement our methods as subproblem solvers of adaptive cubic regularization methods, and numerical results show that our algorithms are comparable to the state-of-the-art algorithms.
\end{abstract}

\section{Introduction}
Motivated by applications in machine learning and signal processing, optimization problems of the following form have attracted significant attention:
\begin{equation} \label{opt:nonconvex}
\min_{x \in \mathbb{R}^n} F(x) ,
\end{equation}
where $F$ is a twice continuously differentiable function that is possibly non-convex.
The cubic regularization method \cite{nesterov2006cubic,cartis2011adaptive} is among the most successful algorithms for solving problem~\eqref{opt:nonconvex}. 
At each iteration of the cubic regularization method, the subproblem takes the form
\begin{equation}\label{opt:CRS}\tag{CRS}
\min_{x \in \mathbb{R}^n}  f_1(x):=\frac{1}{2} x^TAx+b^Tx+\frac{\rho}{3}\norm{x}^3,
\end{equation}
where $\norm{\cdot}$ denotes the Euclidean norm, $A$ is an $n\times n$ symmetric matrix (not necessarily positive semidefinite) and $\rho$ is a regularization parameter.
In particular, $A$ and $b$ represent the Hessian and gradient of the function $F$ at the current iterate, respectively.
It was first proved by Nesterov and Polyak \cite{nesterov2006cubic} that the cubic regularization method enjoys an iteration complexity of $O(\epsilon^{-3/2})$
if each subproblem is solved exactly. Cartis et al.~\cite{cartis2011adaptive} developed a generalization of the cubic regularization method, called ARC, which allows the subproblems to be solved inexactly and the regularization parameter $\rho>0$ to be chosen adaptively. In the same paper, they showed that the iteration complexity of ARC is again $O(\epsilon^{-3/2})$. Complementing to these global complexity results, Yue et al.~\cite{yue2019quadratic} showed that the cubic regularization method enjoys a local quadratic convergence rate under an error bound-type condition.

Despite the above strong theoretical guarantees, the practical performance of the cubic regularization method depends critically on the efficiency of solving its subproblems. As such, there have been considerable endeavors on developing fast algorithms for solving~\eqref{opt:CRS}. One of the most successful algorithms for solving large-scale instances of \eqref{opt:CRS} in practice is the Krylov subspace method~\cite{cartis2011adaptive}.
Carmon and Duchi~\cite{carmon2018analysis} provided the first the convergence rate analysis of the Krylov subspace method. In particular, they showed that the Krylov subspace method achieves an $\epsilon$-approximate optimal solution in $O(\epsilon^{-1/2})$ or $O(\sqrt\kappa\log \epsilon^{-1})$ operations (matrix-vector multiplications) in the easy case\footnote{For the problem~\eqref{opt:CRS}, it is said to be in the easy if the optimal solution $x^*$ satisfies  $\rho\|x^*\|>-\lambda_1$, where $\lambda_1$ is the minimum eigenvalue of $A$, and hard case otherwise.}, where
$\kappa$ is the condition number of \eqref{opt:CRS}.
Unfortunately, the Krylov subspace method may fail to converge to the optimal solution when the problem \eqref{opt:CRS} is in the hard case or close to being in the hard case~\cite{carmon2018analysis}.
Carmon and Duchi also showed in another paper \cite{carmon2019gradient} that the gradient descent method is able to converge to the global minimizer if the step size is  sufficiently small, and the convergence rate is $\tilde O(\epsilon^{-1})$ (where $\tilde O(\cdot)$ hides the logarithmic factors). Although, for the problem~\eqref{opt:CRS}, the convergence rate of the gradient descent method is worse than that of the Krylov subspace method, it works in both the easy and hard cases. On the other hand, based on the cubic regularization method, Agarwal et al.~\cite{agarwal2017finding} derived an algorithm with $\tilde O({\epsilon^{-7/4}})$ operations for finding an approximate local minimum of problem \eqref{opt:nonconvex},  i.e.,  a point $x\in\mathbb{R}^n$ satisfying
\begin{equation*}
\label{eq:station}
\|\nabla F(x)\|\le\epsilon \quad \text{and}\quad \nabla^2 F(x) + \sqrt\epsilon I\succeq 0 ,
\end{equation*}
where $I$ denotes the identity matrix of appropriate dimension and, for any symmetric matrix $M$, the inequality $M \succeq 0$ means that $M$ is positive semidefinite.
A key component of their result is an algorithm for computing an approximate solution to the problem \eqref{opt:CRS} in $\tilde O(\epsilon^{-1/4})$ operations.
However, the approximate solution returned by this algorithm is not an $\epsilon$-approximate global minimizer of the problem~\eqref{opt:CRS} in the traditional sense (see \cite[Theorem 2]{agarwal2017finding} for details). Furthermore, the algorithm in \cite{agarwal2017finding} for solving \eqref{opt:CRS} requires sophisticated parameter tuning, and no numerical results had been provided in the paper. 
Finally, a Newton-like method for solving problems of the form~\eqref{opt:nonconvex} had been recently developed by Birgin and Mart{\'\i}nez~\cite{birgin2019newton}. Each subproblem of their algorithm, which is similar to but not the same as \eqref{opt:CRS}, is constructed and can be efficiently solved by using the so-called mixed factorization (see \cite[Section 2]{birgin2019newton} for details) of the (approximate) Hessian of $F$ at the current point. Birgin and Mart{\'\i}nez~\cite{birgin2019newton} advocated in particular the mixed factorization obtained from the Bunch-Parlett-Kaufman factorization~\cite{golub2012matrix}, a matrix factorization whose computational cost is similar to that of the Cholesky factorization.

From the above discussion, it is desirable to have an algorithm for solving the problem \eqref{opt:CRS} that works efficiently in practice for both the hard and easy cases and enjoys theoretical guarantees. In this paper, we achieve this goal by developing a first-order method for solving arbitrary instances of \eqref{opt:CRS} with  $\tilde O(\epsilon^{-1/2})$ matrix-vector multiplications.
%
Our approach is based on a novel reformulation of the problem~\eqref{opt:CRS}, which is a constrained convex optimization problem built using the minimum eigenvalue of the matrix $A$. The feasible region of the reformulation admits an efficient, closed-form projection. Therefore, when the exact computation of the minimum eigenvalue is viable, we can apply any algorithm for solving constrained convex optimization problems to solve the reformulation to global optimality. The optimal solution to the problem~\eqref{opt:CRS} can then be constructed by using the optimal solution of the reformulation. In practice, it is often prohibitively expensive to compute the exact minimum eigenvalue of the matrix $A$, if not impossible. We circumvent this limitation by developing a surrogate problem to the reformulation. The surrogate problem is again a constrained convex optimization problem with an easily computable projection onto its feasible region. More importantly, the surrogate problem requires only an approximate minimum eigenvalue, which can be computed efficiently by using, e.g., the Lanczos method \cite{golub2012matrix}. Similarly, an $\epsilon$-approximate optimal solution of the problem~\eqref{opt:CRS} can be constructed from an $\epsilon$-approximate solution of the surrogate problem.

The said bound $\tilde O(\epsilon^{-1/2})$ on the number of operations is proved by combining the follow two ideas. First, for any $\delta \in (0,1)$, the Lanczos method returns an $\epsilon$-approximate minimum eigenvalue in $O(\epsilon^{-1/2}\log(n/\delta))$ matrix-vector multiplications with probability at least $1 - \delta$.
Second, solving the surrogate problem by the Nesterov's accelerated projected gradient descent method \cite{nesterov1983method,beck2009fast} (APG) requires $O(\epsilon^{-1/2})$ iterations, where each iteration consists of one gradient and Hessian evaluations and one matrix-vector multiplication.
Therefore, the total number of operations of our method is bounded by $O(\epsilon^{-1/2}\log(n/\delta))$ (see Theorem \ref{thm:err}).  This bound is similar to the sublinear bound for the Krylov subspace method proved in \cite{carmon2018analysis} in the easy case and better than that of the gradient descent method in \cite{carmon2019gradient}. Note also that our bound is for the subproblem and hence not directly comparable with that of \cite{agarwal2017finding}. Besides, our algorithm has the advantage that it is easily implementable. Furthermore, as we shall see in our numerical section, the proposed algorithm works efficiently in practice for high-dimensional problems---our algorithm shows a comparable performance to the Krylov subspace method in the easy case. An another advantage of our algorithm is that, unlike the Krylov subspace method, it works in both the easy and hard cases. This saves us from the computational overhead due to the need of detecting the hard case.

We remark that our approach is inspired by the recent line of research \cite{ho2017second,wang2017linear} on linear-time algorithms for the trust region subproblem
\begin{equation} \label{opt:TRS}\tag{TRS}
\begin{array}{c@{\quad}l}
\displaystyle\min_{x\in \mathbb{R}^n} & \displaystyle \frac{1}{2}x^TAx+b^Tx \\
\noalign{\smallskip}
\mbox{subject to} & \displaystyle \|x\|^2 \le 1,
\end{array}
\end{equation}
and the close resemblance between the problems~\eqref{opt:CRS} and \eqref{opt:TRS}.
More specifically, the algorithms in \cite{ho2017second,wang2017linear} are based on a convex reformulation for the \eqref{opt:TRS} derived in \cite{flippo1996duality}.
Motivated by the works~\cite{ho2017second,wang2017linear}, Jiang and Li~\cite{jiang2019novel} recently derived a novel convex reformulation for the generalized trust region subproblem, which further inspires us to explore hidden convexity for \eqref{opt:CRS} in this paper.
It should also be pointed out that our reformulation and its surrogate problem offer great potential and flexibility for the design of fast algorithms to solve the problem~\eqref{opt:CRS}. Indeed, one can apply any algorithm for constrained convex optimization problems to solve these two optimization problems. Proving theoretical guarantees for other algorithms for solving these two models is left as a future research.

The remaining of this paper is organized as follows. In Section 2, we derive our convex reformulation based on the minimum eigenvalue of matrix $A$ and discuss the computation of the projection to its feasible region. In Section 3, we present a surrogate problem for \eqref{opt:CRS}  and theoretically analyze the complexity of our method when applying the APG to solve the surrogate problem with an approximate minimum eigenvalue computed by the Lanczos method. In Section 4, we first compare the numerical performance of our methods with the Krylov subspace method and then compare our methods against others as a subproblem solver for ARC. We conclude our paper in Section 5.

\section{Convex reformulation}
We first record the optimality condition of \eqref{opt:CRS} \cite{nesterov2006cubic,cartis2011adaptive}, which is given by the following system of equations in $x$ and $\lambda$:
\begin{equation}
\label{eq:optCRS}
Ax+b+\lambda x=0,\quad  A+\lambda I\succeq 0,\quad \text{and}\quad \lambda=\rho \|x\|.
\end{equation}
This optimality condition will be frequently used in this paper.
It is obvious that \eqref{opt:CRS} is equivalent to the following problem:
\begin{equation} \label{opt:RP}\tag{RP}
\begin{array}{c@{\quad}l}
\displaystyle\min_{x\in \mathbb{R}^n, y\in \mathbb{R}} & \displaystyle \frac{1}{2} x^TAx+b^Tx+\frac{\rho}{3}y^\frac{3}{2} \\
\noalign{\smallskip}
\mbox{subject to} & \displaystyle\norm{x}^2\leq y.
\end{array}
\end{equation}
Note that the feasible region $\{(x,y)\in \R^n \times\R: \|x\|^2 \le y\}$ of the problem~\eqref{opt:RP} is convex. Therefore, when $A\succeq 0$, \eqref{opt:RP} is a convex optimization problem and can be solved efficiently by various methods, e.g., APG or projected Barzilai-Borwein method (BBM) \cite{barzilai1988two,tropp2010computational}.
Hence, from now on, we assume that the minimum eigenvalue of matrix $A$, denoted by $\lambda_1$, is negative, i.e., $\lambda_1 < 0$.
Consider the optimization problem
\begin{equation} \label{opt:CP}\tag{CP}
\begin{array}{c@{\quad}l}
\displaystyle\min_{x\in \mathbb{R}^n, y\in \mathbb{R}} & \displaystyle f_2(x,y):=\frac{1}{2} x^T(A-\lambda_1 I) x+b^Tx+\frac{\rho}{3}y^\frac{3}{2}+\frac{\lambda_1 }{2}y \\
\noalign{\smallskip}
\mbox{subject to} & \displaystyle \norm{x}^2\leq y.
\end{array}
\end{equation}
Problem \eqref{opt:CP} is a convex  problem because $f_2$ is separable in $x$ and $y$
and is convex in each of these two variables. The following theorem shows that problem~\eqref{opt:CRS} is equivalent to problem~\eqref{opt:CP}.
\begin{thm}\label{thm:cvreform}
Problem \eqref{opt:CRS} is equivalent to \eqref{opt:CP} in the following sense. First, the two problems have the same optimal value. Second, if $x^*$ is an optimal solution to \eqref{opt:CRS}, then
$(x^*,\|x^*\|^2)$ is an optimal solution to \eqref{opt:CP}. Third, if $(\tilde x,\tilde y)$ is an optimal solution to \eqref{opt:CP}, then an optimal solution to \eqref{opt:CRS} is given by
\begin{equation*}
\hat{x} =
\begin{cases}
\tilde{x} & \text{if } \|\tilde x\|^2=\tilde y,\\
\tilde{x} + \zeta v & \text{if } \|\tilde x\|^2 < \tilde{y},
\end{cases}
\end{equation*}
where $\zeta $ is a root of the quadratic equation $\|\tilde x+\zeta v\|^2=\tilde y$ and $v$ is an eigenvector associated with $\lambda_1$.
\end{thm}
\begin{proof}
Denote by $\text{Val}\eqref{opt:CRS}$ and $\text{Val}\eqref{opt:CP}$ the optimal values of problems \eqref{opt:CRS} and \eqref{opt:CP}, respectively. We first observe that \eqref{opt:CP} is a convex problem and satisfies the Slater condition. Assume that $x^*$ is an optimal solution to \eqref{opt:CRS}. By using the optimality condition \eqref{eq:optCRS}, we can easily show that the triplet $(x,y , \mu) = (x^*,\|x^*\|^2, \tfrac{1}{2} (\rho\|x^*\|+\lambda_1))$ satisfies the KKT system of \eqref{opt:CP}:
\begin{equation}\label{eq:CP_KKT}
(A-\lambda_1I)x+b+2\mu x=0 \text{ and}~\frac{\rho}{2}y^\frac{1}{2}+\frac{\lambda_1}{2}-\mu=0
\end{equation}
This implies that $(x^*,\|x^*\|^2) $ is an optimal solution to \eqref{opt:CP} and that $\text{Val}\eqref{opt:CRS} \geq \text{Val}\eqref{opt:CP}.$ On the other hand, because of the assumption $\lambda_1 < 0 $ and the constraint $\|x\|^2 \le y$, we have that $\text{Val}\eqref{opt:CRS} \le \text{Val}\eqref{opt:CP}.$ Therefore, $\text{Val}\eqref{opt:CRS} = \text{Val}\eqref{opt:CP}.$ This completes the proof of the first and second claims.

To prove the third claim, assume that \eqref{opt:CP} has an optimal solution $(\tilde x,\tilde y)$. Suppose  $\mu$ is a Lagrangian multiplier associated with the constraint in \eqref{opt:CP}.  If $\|\tilde x\|^2=\tilde y$, from the KKT system~\eqref{eq:CP_KKT}, we have that
\begin{equation}\label{eq:pf_1}
A\tilde x-\lambda_1 \tilde x+b+2\mu \tilde x=0
\end{equation}
and
\begin{equation}\label{eq:pf_2}
\frac{1}{2}\rho\sqrt{\tilde y}+ \frac{1}{2}\lambda_1 -\mu=0 .
\end{equation}
Equation~\eqref{eq:pf_2} implies
$\mu=\rho\sqrt{\tilde y}/2+\lambda_1/2\ge0$. This, together with $\|\tilde x\|^2=\tilde y$ and  $A\tilde x-\lambda_1 \tilde x+b+2 \mu \tilde x=0$, implies that  $A\tilde x+b+\bar\lambda \tilde x=0$ and $A+\bar\lambda I\succeq0$, for $\bar\lambda=2 \mu -\lambda_1=\rho\|\tilde x\|.$ Hence, due to \eqref{eq:optCRS},  $\tilde x$ is also optimal for \eqref{opt:CRS} and the objective values of \eqref{opt:CRS} and \eqref{opt:CP} are the same due to   $\|\tilde x\|^2=\tilde y$.

Next, we consider the case of  $\|\tilde x\|^2<\tilde y$. Let $v$ be an eigenvector of matrix $A$ associated with the minimum eigenvalue $\lambda_1$. By complementary slackness, $ \mu =0$. Then, equation~\eqref{eq:pf_1} implies that $b^Tv=0$.  Hence, there exists $\zeta$ such that $\|\tilde x+ \zeta v\|=\sqrt{\tilde y}$ and $(\tilde x+\zeta v,\tilde y)$ is still a solution to \eqref{opt:CP}. Using the same argument for the case of $\|\tilde x\|^2=\tilde y$, we can show that $\tilde x+\zeta v$ is an optimal solution for \eqref{opt:CRS}. This completes the proof.
\end{proof}

Optimization problems of the form
\begin{equation}\label{opt:con_com}
\min_{x\in \R^n, y\in \R} g(x,y)+h(x,y),
\end{equation}
where $g$ is a smooth convex function and $h$ is a non-smooth convex function, are called  convex composite minimization problems.
Letting $S=\{(x,y):\norm{x}^2\leq y\}$,
problem \eqref{opt:CP} can be written as a convex composite minimization problem:
\begin{equation*}
\min_{x\in \R^n, y\in \R} f_2(x,y)+ \iota_S(x,y),
\end{equation*}
where $\iota_S$ is the indicator function $$\iota_S(x,y):=\left\{
\begin{array}{ll}0, &\text{if }(x,y)\in S,\\
+\infty, &\text{otherwise}.\end{array}\right.$$
General convex composite minimization problems~\eqref{opt:con_com} can be solved by many different algorithms such as  APG, BBM, proximal quasi-Newton methods \cite{Ghanbari2018} and proximal Newton methods~\cite{yue2019family}. 
In order to apply these methods, we need to efficiently compute the proximal mapping with respect to the non-smooth function $h$ in \eqref{opt:con_com}. In our situation, $h = \iota_S$ and hence the proximal mapping reduces to the orthogonal projection $ \Pi_S(x,y)$ onto the closed convex set $S$, i.e.,
\begin{equation*}
 \Pi_S(x,y) = \argmin_{(x',y') \in S} \| (x',y') - (x,y) \|^2 .
\end{equation*}
The following theorem shows that such a projection can be done in $O(n)$ time.
\begin{thm}
\label{thm:pj2S}
For any point $(x_0, y_0) \in \R^n \times \R$,
the projection $\Pi_S(x_0,y_{0})$ is given by
\begin{equation}
\Pi_S(x_0,y_{0}) =
\begin{cases}
\displaystyle \left(x_0,y_0 \right), & \text{if } \|x_0\|^2 \leq y_0,\\
\displaystyle \left(\frac{x_0}{1+\mu^*} ,\, y_0+\frac{\mu^*}{2}\right), & \text{otherwise,}
\end{cases}
\end{equation}
where $\mu^{*}$ is the unique  solution in the interval  $[\max\{0,-2y_0\},\infty)$ of the univariate cubic equation
\begin{equation}\label{eq:cubic}
\frac{1}{2}\mu^3+(y_0+1)\mu^2+(2y_0+\frac{1}{2})\mu-x_0^Tx_0+y_0=0.
\end{equation}
\end{thm}
\begin{proof}
The case of $x_0^Tx_0\leq y_0$ is trivial.
So, we consider the case that  $x_0^Tx_0> y_0$. The projection is defined as the solution to the (strongly convex) optimization problem
\begin{equation} \label{opt:proj}
\begin{array}{c@{\quad}l}
\displaystyle\min_{x\in \mathbb{R}^n, y\in \mathbb{R}} & \displaystyle \norm{(x,y)-(x_0,y_0)}^2 \\
\noalign{\smallskip}
\mbox{subject to} & \displaystyle \norm{x}^2\leq y.
\end{array}
\end{equation}
The KKT optimality condition of problem~\eqref{opt:proj} can be written as
\begin{eqnarray}
&&2(x-x_{0})+2\mu x=0, \label{PKKT1}\\
&&2y-2y_0-\mu=0,\label{PKKT2}\\
&&\mu(\norm{x}^2-y)=0,\notag\label{PSC}\\
&&\norm{x}^2\leq y,\notag\\
&&\mu\geq0.\notag
\end{eqnarray}
We have $x=\frac{x_0}{1+\mu}$ and $y=y_0+\frac{\mu}{2}$ from (\ref{PKKT1}) and (\ref{PKKT2}), respectively.
Suppose that $\mu=0$. The  optimality condition reduces to $x=x_0$ and $y=y_0$, which contradicts to the constraint $\norm{x}^2>y$ of problem~\eqref{opt:proj}. Therefore, we have
$\mu>0$ and hence $\norm{x}^2=y$ by complementary slackness. This leads to the univariate cubic equation
\begin{equation*}
\left( \frac{x_0}{1+\mu} \right)^T\frac{x_0}{1+\mu}=y_0+\frac{\mu}{2} ,
\end{equation*}
which is equivalent to \eqref{eq:cubic} and implies, in particular, that $2y_0+\mu\ge0$.
Define $$h(\mu)=\frac{1}{2}\mu^3+(y_0+1)\mu^2+ \left( 2y_0+\frac{1}{2} \right)\mu-x_0^Tx_0+y_0 .$$
Since $2y_0+\mu\ge0$ and $\mu\ge0$, the derivative $h'$ satisfies
$$h'(\mu)=\frac{3}{2}\mu^2+2(y_0+1)\mu+\left( 2y_0+\frac{1}{2} \right)=\frac{1}{2}\mu^2+(2y_0+\mu)\mu+(2y_0+2\mu)+\frac{1}{2}\ge \frac{1}{2}.$$
Hence  $h(\mu)$ is strictly increasing on $[\max\{0,-2y_0\}, \infty)$.
Observing that  $h(0)=y_0-x_0^Tx_0<0$, $h(-2y_0)=-x_0^Tx_0<0$ and $h(+\infty)=+\infty$, there exists exactly one root in the interval $[\max\{0,-2y_0\},\infty)$.
Denote the solution of equation $h(\mu)=0$ in this interval by $\mu^*$. Then, we have $$x=\frac{x_0}{1+\mu^{*}}\quad \text{and} \quad y=y_0+\frac{\mu^{*}}{2}, $$ which completes the proof.
\end{proof}
In practice, to find a root of the cubic equation \eqref{eq:cubic} in the interval  $[\max\{0,-2y_0\},\infty)$, we use a hybrid method obtained by combining the bisection method and the Newton's method.
Numerically, our hybrid method is faster and more stable than the function \texttt{roots} in MATLAB.
The projection can be done in runtime $O(n)$ as formulating the cubic equation cost $O(n)$ and solving the univariate cubic equation costs $O(1)$.

\section{Complexity to achieve an $\epsilon$-optimal solution of \eqref{opt:CRS}}

\subsection{Another Equivalent Convex Reformulation}
To achieve a theoretical complexity for solving convex composite optimization problem~\eqref{opt:con_com} with first-order methods such as APG~\cite{nesterov1983method}, the function $g$ is often required to have a Lipschitz continuous gradient on its domain ${\rm dom}(g)$,   i.e., there exists a constant $L>0$ such that
\begin{equation*}
\label{eq:lip}
\|\nabla g(x)-\nabla g(y)\| \le L\|x-y\| ,~\forall x,y\in {\rm dom}(g).
\end{equation*}
However, one can easily check that the gradient $\nabla f_2$ of the objective $f_2$ of \eqref{opt:CP} is not Lipschitz continuous at those points $(x,y)$ with $y = 0$.
To remedy this, instead of \eqref{opt:CP}, we consider the following problem, which ensures $y$ is bounded below from 0 by imposing an extra constrain $y \ge l$:
\begin{equation} \label{opt:BCP}\tag{BCP}
\begin{array}{c@{\quad}l}
\displaystyle\min_{x\in \mathbb{R}^n, y\in \mathbb{R}} & \displaystyle f_2(x,y) \\
\noalign{\smallskip}
\mbox{subject to} & \displaystyle \norm{x}^2\leq y,\ y\ge l,
\end{array}
\end{equation}
%
%
where $l=\lambda_1^2/\rho^2$. To justify the choice of the lower bound $l$ in \eqref{opt:BCP}, we note that the function $\frac{\rho}{3}y^\frac{3}{2}+\frac{\lambda_1}{2}y$ is decreasing when $\sqrt y \le -\lambda_1/\rho$. Therefore, any optimal solution $(\tilde{x}, \tilde{y})$ of \eqref{opt:CP} must satisfy $\tilde{y} \ge (-\lambda_1/\rho)^2 = l$, and hence problem \eqref{opt:BCP} has the same objective value and optimal solutions as problem~\eqref{opt:CP}.


Problem~\eqref{opt:BCP} is again in the form of a convex composite minimization problem~\eqref{opt:con_com}. Denote by $B=\{(x,y)\in\R^n\times \R: \norm{x}^2\le y,y\ge l\}$ the feasible region of problem~\eqref{opt:BCP}.
The next theorem shows that the projection $\Pi_B$ onto the feasible region $B$ is again easily computable.

\begin{thm}
\label{thm:pj2B}
For any point $(x_0, y_0) \in \R^n \times \R$,
the projection $\Pi_B(x_0,y_{0})$ is given by
\[
\Pi_B(x_0,y_{0})=\left\{\begin{array}{ll}(x_1,y_1)&\text{if $ y_1\ge l$},\\ (x_0,l)&\text{if $y_1<l$ and $\|x_0\|<\sqrt l$,}\\
(\sqrt l x_0/\|x_0\|,l) &\text{otherwise,}\end{array}\right.
\]
where $(x_1,y_1)=\Pi_S(x_0,y_0)$.
\end{thm}
\begin{proof}
Let $(x_2,y_2)$ be the projection of $(x_0,y_0)$ onto $B$.
If $y_1\ge l$, then $(x_1,y_1) =\Pi_S(x_0,y_0)$ is the solution to the
problem
\begin{equation*}
\begin{array}{c@{\quad}l}
\displaystyle\min_{x\in \mathbb{R}^n, y\in \mathbb{R}} & \displaystyle \norm{(x,y)-(x_0,y_0)}^2 \\
\noalign{\smallskip}
\mbox{subject to} & \displaystyle \norm{x}^2\leq y,\ y\ge l.
\end{array}
\end{equation*}
Next, we consider the case of $y_1<l$. In this case, we must have $y_2=l$ since otherwise $(x_2,y_2)$ is also the projection of $(x_0,y_0)$ onto $S$, which contradicts with $y_1<l$. Hence, $x_2$ is actually the solution to the problem
\begin{equation*}
\begin{array}{c@{\quad}l}
\displaystyle\min_{x\in \mathbb{R}^n} & \displaystyle \norm{x -x_0}^2 \\
\noalign{\smallskip}
\mbox{subject to} & \displaystyle \norm{x}^2\leq  l.
\end{array}
\end{equation*}
We thus have the following two implications: if $\|x_0\|< \sqrt l$, then $x_{2}=x_0$; and if $\|x_0\|\ge \sqrt l$, then $x_{2}=\sqrt lx_0/\|x_0\|$. This completes the proof.
\end{proof}
\noindent For Theorem~\ref{thm:pj2B}, the projection onto $B$ is as cheap as the projection onto $S$ because the former costs at most two more scalar comparisons, which are negligible, than
the latter (note that $\|x_0\|$ is already computed in the computation of the projection onto $S$).

\subsection{A Surrogate Problem}
When the dimension $n$ is high, the exact computation of the minimum eigenvalue is prohibitively expensive, if not impossible. For computational efficiency, an approximate eigenvalue is preferred when only an approximate solution of \eqref{opt:CRS} is needed, which is often the case in practice.
When an approximate minimum eigenvalue $\theta \approx \lambda_1$ is used in the problem~\eqref{opt:BCP}, the objective $\frac{1}{2} x^T(A-\theta I) x+b^Tx+\frac{\rho}{3}y^\frac{3}{2}+\frac{\theta}{2}y$ could be non-convex. Therefore, we need to slightly modify the problem~\eqref{opt:BCP}.
Let the approximate minimum eigenvalue $\theta$ satisfies $\lambda_1\le\theta\le \lambda_1+\epsilon$ and define $\eta :=-\theta+\epsilon+\lambda_1\ge0.$
Noting that $-\theta+\epsilon=-\lambda_1+\eta$ (we will frequently use this equality in subsequent analysis), we obtain the following problem as a surrogate problem to \eqref{opt:CRS}:
\begin{equation} \label{opt:BCP_p}\tag{SP}
\begin{array}{c@{\quad}l}
\displaystyle\min_{x\in \mathbb{R}^n, y\in \mathbb{R}} & \displaystyle f_3(x,y) := \frac{1}{2} x^T\left(A+(-\theta+\epsilon)I\right) x+b^Tx+\frac{\rho}{3}y^\frac{3}{2}-\frac{-\theta+\epsilon}{2}y \\
\noalign{\smallskip}
\mbox{subject to} & \displaystyle \norm{x}^2\leq y,\ y\ge \hat{l},
\end{array}
\end{equation}
where $\hat{l}=(-\theta+\epsilon)^2/\rho^2$. To justify the lower bound $\hat{l}$ for $y$, we note that $\frac{\rho}{3}y^\frac{3}{2}-\frac{-\theta+\epsilon}{2}y$ is decreasing when $y\le \hat{l}$, and hence $\hat{l}$ is a lower bound for any optimal $y$. From now on, we denote by $\hat{B}:=\{(x,y): \norm{x}^2\leq y,~y\ge \hat{l}\}$ and $(x^\eta,y^\eta)$ the feasible region and an optimal solution to \eqref{opt:BCP_p}, respectively. By Theorem~\ref{thm:pj2B}, the feasible region $\hat{B}$ admits an easily computable projection.

Our theoretical convergence rate of solving problem~\eqref{opt:CRS} is based on the surrogate problem~\eqref{opt:BCP_p}. Specifically, we shall specialize the backtracking line search version of APG~\cite{beck2009fast} to problem~\eqref{opt:BCP_p} (see Algorithm~\ref{alg:APG}) and show in Theorem~\ref{thm:err} below that the sequence of iterates converges sublinearly to an optimal solution of problem~\eqref{opt:BCP} (which is also an optimal solution to problem~\eqref{opt:CP}). In view of Theorem~\ref{thm:cvreform}, a convergence rate for solving \eqref{opt:CRS} is thus obtained.
It should be pointed out that, unlike the original APG, we reset the final solution returned by APG (in Lines 8--12 of Algorithm~\ref{alg:APG}) to achieve an equal or smaller objective value  (see the proof in Theorem \ref{thm:err}).

\textbf{Remark:} If we directly use the approximate minimum eigenvalue $\theta$ to replace the exact minimum eigenvalue $\lambda_1$ in \eqref{opt:BCP},  we get the following  problem:
\begin{equation} \label{opt:AP}\tag{AP}
\begin{array}{c@{\quad}l}
\displaystyle\min_{x\in \mathbb{R}^n, y\in \mathbb{R}} & \displaystyle \frac{1}{2} x^T(A-\theta I) x+b^Tx+\frac{\rho}{3}y^\frac{3}{2}+\frac{\theta}{2}y \\
\noalign{\smallskip}
\mbox{subject to} & \displaystyle \norm{x}^2\leq y,\ y\ge l,
\end{array}
\end{equation}
In Appendix~\ref{app:b}, we show that solving \eqref{opt:AP} yields an approximate optimal solution to \eqref{opt:CRS} if $\epsilon$ is sufficiently small, i.e., the eigenvalue computation is sufficiently accurate.
We also show in Appendix~\ref{app:b} that either all the stationary points, which are approximate optimal solutions  of \eqref{opt:AP}, share the same objective value,  or there is a unique  stationary point that is the optimal solution of \eqref{opt:AP} if $-\theta>\bar\lambda$, where $\bar{\lambda}$ is some constant such that $\bar\lambda<-\lambda_1$.
Note that when $\epsilon\le-\lambda_1-\bar\lambda$, we always have that $\theta<\lambda_1+\epsilon<-\bar\lambda$ and hence that $-\theta>\bar\lambda$.
However, the constant $\bar\lambda$ is unknown a priori and hence our formulation  \eqref{opt:AP} may have a non-optimal stationary point if we choose a $\theta$ that is not close enough to $\lambda_1$. This is why we focus on \eqref{opt:BCP_p} in this paper. Nevertheless,  we will  compare the empirical performance between \eqref{opt:BCP_p}  and \eqref{opt:AP} in the numerical section.
\begin{algorithm}[!t]
\begin{algorithmic}[1]
\caption{APG  for \eqref{opt:BCP_p} }
\label{alg:APG}
\Require $f_3$, $\nabla f_3$, $L_0>0,\xi>1$, $\epsilon>0$, $\theta<0$, $x_0 \in \R^n$ and $y_0\in \R$.
\State choose $\beta_1=\alpha_0 = (x_0^T,y_0)^T$ and $t_1=1$
\For {$k=1,2,...,k_{\max}$}
\State find the smallest non-negative integer $i_k$ such that  $\bar L=\xi^{i_k}L_{k-1}$  and
\[ f_3(\alpha_k)\ge  f_3(\beta_k)+\nabla  f_3(\beta_{k})^T(\alpha_k-\beta_k)+\frac{\bar L}{2}\|\alpha_k-\beta_k\|^2, \] where $\alpha_k=\Pi_{\hat{B}}(\beta_k-\frac{1}{\bar L}\nabla f_3(\beta_k))$
\State set $L_k=\xi^{i_k}L_{k-1}$
\State compute $t_{k+1}=\frac{1+\sqrt{1+4t_k^2}}{2}$
\State compute $\beta_{k+1}=\alpha_k+\left(\frac{t_k-1}{t_{k+1}}\right)(\alpha_k-\alpha_{k-1})$
\EndFor
\If {$\alpha_k(n+1)>\|\alpha_k(1:n)\|^2$ and $\sqrt{\alpha_k(n+1)}>(-\theta+\epsilon)/\rho$}
\State set $x_k=\alpha_k(1:n)$ and $y_{k}=\max\{\|\alpha_k(1:n)\|^2, (-\theta+\epsilon)^{2}/\rho^{2}\}$
\Else
\State set $(x_k^T,y_k)^T=\alpha_k$
\EndIf
\end{algorithmic}
\end{algorithm}

\subsubsection{Approximate Computation of Eigenpairs}
To obtain an approximate eigenpair, we recall the Lanczos method for approximately finding the minimum eigenvalue and its associated eigenvector \cite{golub2012matrix}.
The Lanczos method achieves a fast complexity bound for eigenvalue computation \cite{kuczynski1992estimating} and is an important component for proving complexity bounds for non-convex unconstrained optimization in the literature \cite{agarwal2017finding,carmon2018accelerated,royer2018complexity}. The specific result on the Lanczos method we need is the following lemma.
\begin{lem}[\cite{kuczynski1992estimating} and Lemma 9 in \cite{royer2018complexity}]
\label{lem:lanczos}
Let $H$ be a symmetric matrix satisfying $\|H\|_2\le U_{H}$ for some $U_{H}>0$, where $\|\cdot \|_2$ denotes the operator 2-norm of a matrix, and $\lambda_1$ its minimum eingenvalue.
Suppose that the Lanczos procedure is applied to find the largest eigenvalue of $ U_{H}I-H$
starting at a random vector distributed uniformly over the unit sphere. Then, for any
$\epsilon>0$ and $\delta\in(0, 1)$, there is a probability at least $1-\delta$ that the procedure outputs a
unit vector $v$ such that
$v^THv\le\lambda_1 +\epsilon$
in at most
$\min\left\{n,\frac{\log(n/\delta^2)}{2\sqrt2}\sqrt{\frac{ U_{H}}{\epsilon}}\right\}$
iterations.
\end{lem}

\subsubsection{Convergence Rate of APG for \eqref{opt:BCP_p}}
We first collect some basic properties of APG.
\begin{lem}[\cite{nesterov1983method,beck2009fast}]
\label{lem:APG}
Consider a function $G(x)=g(x)+h(x)$, where
$g$ is continuously differentiable, convex function with the gradient $\nabla g$ being $L$-Lipschitz continuous on its domain ${\rm dom}(g)$ and
 $h$ is a proper, closed, and convex function that can possibly be non-smooth. Let $\{x_k\}_{k=1}^\infty$ be the sequence generated by APG. Then, we have
$$G(x_k)-G^*\le \frac{2\xi L\|x^*-x_0\|^2}{(k+1)^2},$$ where $x^*$ is an optimal solution and $G^*$ is the optimal value of $G(x)$. Equivalently,
in order to guarantee $G(x_k)-G^*\le\epsilon$, we need at most
$k=\sqrt{2\xi L}\|x^*-x_0\| \epsilon^{-1/2}-1$ iterations.
\end{lem}

Restricting the objective function $f_3$ in \eqref{opt:BCP_p} to the set $\hat{B}$, the gradient $\nabla f_3$ is then $\gamma$-Lipschitz continuous, where
\begin{equation}\label{eq:gamma}
\gamma=\max\left\lbrace\|A+( \epsilon-\theta ) I\|_2,\frac{\rho}{4\sqrt{ \hat{l} }}\right\rbrace .
\end{equation}
Applying Lemma~\ref{lem:APG} to problem \eqref{opt:BCP_p} with $g = f_3$ and $h = \iota_{\hat{B}}$, we obtain that $$f_3(x_k,y_k) - f (x^\eta, y^\eta) \le\epsilon$$
after at most  $k=\sqrt{2\xi \gamma}\sqrt{\|x^\eta-x_0\|^2+ (y^\eta - y_0 )^2 }\epsilon^{-1/2}-1$ iterations.

The next theorem shows that with probability at least $1 - \delta$, our algorithm returns an $\epsilon$-approximate optimal solution to problem~\eqref{opt:CRS} using at most $O(\epsilon^{-1/2}\log(n/\delta))$ operations (including those in the approximate eigenpiar computation and the APG).


\begin{thm}\label{thm:err}
Let $X^*$ be the optimal solution set, $(x^\eta,y^\eta)$ be any optimal solution to problem \eqref{opt:BCP_p}, $R=\inf_{(x, y) \in X^*}\|(x, y) - (x_0,y_0)\|$ the initial distance to the optimal solution, $(x^*,y^*)$ an optimal solution to problem~\eqref{opt:BCP} with $\|x^*\|^2=y^*$ (which always exists) and $(x_{k},y_{k})$ the solution returned by Algorithm~\ref{alg:APG}, where $k\ge\sqrt{2\xi\gamma}R\epsilon^{-1/2}-1$ and $\gamma$ is as defined in \eqref{eq:gamma}. Define
$$\tilde x=
\left\{\begin{array}{ll}x_{k}, &\text{if }\|x_{k}\|^2=y_k,\\
x_k+tv, &\text{otherwise,}\end{array}\right.$$
 where $v$ is an approximate eigenvector that satisfies $v^TAv\le\lambda_1+\epsilon$ and $\|v\|=1$, and $t$ is chosen such that $t(v^TAx_k+b^Tv+(-\lambda_1+\eta )x_k^Tv)\le0$ and $\|x_k+tv\|^2=y_k$ (which also always exists).
Then, we have  $$f_{1}(\tilde x)-f_1(x^*)\leq\epsilon+({-\lambda_1+\eta)^{2}}\epsilon/{\rho}^2=O(\epsilon),$$ where $f_1$ is the objective function in \eqref{opt:CRS}.
 Furthermore, when the approximate eigenpair is computed by the Lanczos method, the output is correct with probability at least $1-\delta$ and the total number of matrix-vector products is at most $$\sqrt{2\xi\gamma}R\epsilon^{-1/2}-1+\frac{\log(n/\delta^2)}{2\sqrt2}\sqrt{\frac{ \|A\|_2}{\epsilon}}=O(\epsilon^{-1/2}\log(n/\delta)). $$
\end{thm}
\begin{proof}
Recall that $f_2$ and $f_3$ are the objective functions of \eqref{opt:BCP} and \eqref{opt:BCP_p}, respectively.
For any optimal solution $x^*$ of \eqref{opt:CRS}, $(x^*,\|x^*\|^2)$ is an optimal solution of \eqref{opt:BCP}.
Therefore, an optimal solution $(x^*, y^*)$ satisfying $\|x^*\|^2 = y^*$ always exists.
Let $E_{k}=f_3(x_{k},y_{k})-f_3(x^\eta,y^\eta)$.  From Lemma \ref{lem:APG}, we obtain that
$f_{3}(\alpha_k)-f_3(x^\eta,y^\eta)<\epsilon.$ If  $\alpha_k(n+1)>\|\alpha_k(1:n)\| ^2$ and $\sqrt{\alpha_k(n+1)}>(-\lambda_1+\eta)/\rho$, we then go to Line 8 and Algorithm 1 outputs $(x_k,y_k)$ instead of $\alpha_k$.
The $y$-part of the objective function $f_3$, i.e.,
$$\frac{\rho}{3}y^\frac{3}{2}-\frac{-\theta+\epsilon}{2}y,$$
is increasing when $\sqrt y\ge(-\theta+\epsilon)/\rho$, and hence Line 8 outputs a  solution whose objective value is at most $f_3(\alpha_k)$. Hence
 $E_{k}\le \epsilon$ for all $k\ge\sqrt{2\xi\gamma}R\epsilon^{-1/2}-1$.
Using this, we have
\begin{equation}\label{eq:f2bd}
\begin{split}
&\, f_3(x_{k},y_{k})-f_2(x^*,y^*)\\
= &\, f_3(x_{k},y_{k})-f_3(x^\eta,y^\eta)+f_3(x^\eta,y^\eta)-f_3(x^*,y^*)+f_3(x^*,y^*)-f_2(x^*,y^*)\\
\le&\, E_{k}+0+\frac{\eta}{2}(\|x^{*}\|^2-y^*)\\
=&\, E_{k},
\end{split}
\end{equation}
where  the inequality follows from the fact $f_3(x^\eta,y^\eta)-f_3(x^*,y^*)\le0$ because $(x^\eta,y^\eta)$ is an optimal solution to \eqref{opt:BCP_p} and the last equality from the fact that $\|x^*\|^2=y^*$.

If $\|x_k\|^2=y_k$, we have that $\tilde x=x_k$ and hence that $f_3(x_{k},y_{k})=f_1(\tilde x)$.
Substituting $f_3(x_{k},y_{k})=f_{1}(\tilde x)$ to \eqref{eq:f2bd} and noting that $f_1(x^*)=f_2(x^*,y^*)$, we have that $f_{1}(\tilde x)-f_1(x^*)\leq\epsilon.$
If $\|x_k\|^2<y_k$,  we have $\tilde x=x_k+tv$ with $\|\tilde x\|^2=y_k$ and hence
\begin{equation}\label{eq:f13bd}
\begin{split}
&\, f_{1}(\tilde x)-f_3(x_{k},y_{k})\\
=&\, \frac{1}{2} (x_{k}+tv)^TA(x_{k}+tv)+b^T(x_{k}+tv)+\frac{\rho}{3}\norm{(x_{k}+tv)}^3\\
&-\left(\frac{1}{2} x_{k}^TAx_{k}+b^Tx_{k}+\frac{\rho}{3}y_{k}^{3/2}+\frac{-\lambda_1+\eta}{2}(\|x_k\|^2-y_k)\right)\\
=& \, tv^TAx_k+\frac{t^2}{2}v^TAv+tb^Tv-\frac{-\lambda_1+\eta}{2}(\|x_k\|^2-\|x_k+tv\|^2)\\
=& \, t(v^TAx_k+b^Tv)+\frac{t^2}{2}(\lambda_1+\epsilon-\eta)-\frac{-\lambda_1+\eta}{2}(-2tx_k^Tv-t^2)\\
=& \, t(v^TAx_k+b^Tv+(-\lambda_1+\eta )x_k^Tv)+\epsilon t^2/2\\
\le&\, \epsilon t^2/2,
\end{split}
\end{equation}
where the third equality follows from $v^TAv=\theta=\lambda_1+\epsilon-\eta$ and the inequality from $t(v^TAx_k+b^Tv+(-\lambda_1+\eta )x_k^Tv)\le0$.
Note that a constant $t$ satisfying such an inequality always exists. Indeed, since $\|x_k\|^2<y_k$, the equation $\|x_k+tv\|^2=y_k$ (in $t$) have two roots of opposite signs.
Hence, we can always choose a $t$ such that $t(v^TAx_k+b^Tv+(-\lambda_1+\eta )x_k^Tv)\le0.$
Using the  inequalities \eqref{eq:f2bd}, \eqref{eq:f13bd} and the fact that $f_1(x^*)=f_2(x^*,y^*)$, we get
$ f_{1}(\tilde x)-f_1(x^*)\le \epsilon+\epsilon t^2/2$.
Also, $\|x_k+tv\|^2=y_k$ implies that $t\le \|x_k\|+\sqrt{y_k}\le2\left(\frac{-\lambda_1+\eta}{\rho}\right )$.
Thus, we have
$$ f_{1}(\tilde x)-f_1(x^*)\le \epsilon+2\epsilon\left(\frac{-\lambda_1+\eta}{\rho}\right )^2\le\epsilon+2\left(\frac{-\lambda_1+\epsilon}{\rho}\right )^2\epsilon,$$ where the last inequality follows from $0\le \eta:=-\theta+\lambda_1+\epsilon\le\epsilon$.

From Lemma \ref{lem:lanczos}, with probability at least $1-\delta$, such $\theta$ and $v$ can be computed in at most $\frac{\log(n/\delta^2)}{2\sqrt2}\sqrt{\frac{ \|A\|_2}{\epsilon}}$ iterations.
And Lemma \ref{lem:APG} shows that the number of operations required by Algorithm~\ref{alg:APG} is at most $\sqrt{2\xi\gamma}R\epsilon^{-1/2}-1$. This completes the proof.
\end{proof}

\section{Numerical experiments}
In this section, we first compare performance of our subproblem solver to the Krylov subspace method on randomly generated instances whose matrix $A$ in the quadratic term has at least one negative eigenvalue.
We then compare ARC (\cite{cartis2011adaptive}) algorithms with different subproblem solvers on test problems from the CUTEst collection (\cite{gould2015cutest}). 

\subsection{Comparison for subproblem solvers}
\begin{sidewaystable}
  \tiny
  \centering
  \begin{tabular}{c|cccc|cccc|cccc|cccc}
  \toprule

  \multicolumn{17}{c}{$K=10,~n=2000$}\\
  \midrule

  \multirow{2}{*}{Methods}
  &\multicolumn{4}{c|}{$\kappa=10$}
  &\multicolumn{4}{c|}{$\kappa=10^2$}
  &\multicolumn{4}{c|}{$\kappa=10^3$}
  &\multicolumn{4}{c}{$\kappa=10^4$} \\
  \cmidrule(r){2-5} \cmidrule(r){6-9} \cmidrule(r){10-13} \cmidrule(r){14-17}
    &  fval-opt  &iter  & time  & time$_\text{eig}$
    &   fval-opt  &iter  & time  & time$_\text{eig}$
    &   fval-opt  &iter  & time  & time$_\text{eig}$
    &   fval-opt  &iter  & time  & time$_\text{eig}$  \\
  \midrule

  \texttt{BBM(AP)}  & 5.5e-06 & 13.2 & 2.85e-03 & 1.95e-03
  & 8.3e-06 & 52.4 & 3.58e-03 & 1.76e-03
  & 6.9e-06 & 120.4 & 1.11e-02 & 5.79e-03
  & 8.3e-06 & 338.8 & 1.71e-02 & 7.61e-03 \\
  \texttt{BBM(SP)}  & 4.3e-06 & 15.0 & 2.70e-03 & 1.95e-03
  & 6.3e-06 & 42.2 & 3.15e-03 & 1.76e-03
  & 6.8e-06 & 123.8 & 1.08e-02 & 5.79e-03
  & 6.4e-04 & 361.6 & 1.73e-02 & 7.61e-03 \\
  \texttt{APG(AP)} & 4.1e-06 & 15.2 & 2.85e-03 & 1.95e-03
  & 8.6e-06 & 55.4 & 3.77e-03 & 1.76e-03
  & 8.6e-06 & 99.8 & 1.01e-02 & 5.79e-03
  & 2.1e-03 & 276.4 & 1.60e-02 & 7.61e-03 \\
  \texttt{APG(SP)} & 6.4e-06 & 15.2 & 2.71e-03 & 1.95e-03
  & 8.0e-06 & 54.0 & 3.69e-03 & 1.76e-03
  & 9.0e-06 & 108.8 & 9.56e-03 & 5.79e-03
  & 3.2e-03 & 326.0 & 1.75e-02 & 7.61e-03 \\
  \texttt{Krylov} & 6.8e-06 & 9.0 & 1.69e-03 & 0
  & 7.7e-06 & 26.6 & 3.12e-03 &0
  & 9.5e-06 & 61.8 & 7.31e-03 & 0
  & 8.9e-06 & 90.8 & 9.58e-03 & 0 \\
  \midrule
  \midrule

  \multicolumn{17}{c}{$K=10,~n=10000$}\\
  \midrule

  \multirow{2}{*}{Methods}
  &\multicolumn{4}{c|}{$\kappa=10$}
  &\multicolumn{4}{c|}{$\kappa=10^2$}
  &\multicolumn{4}{c|}{$\kappa=10^3$}
  &\multicolumn{4}{c}{$\kappa=10^4$} \\
  \cmidrule(r){2-5} \cmidrule(r){6-9} \cmidrule(r){10-13} \cmidrule(r){14-17}
      &  fval-opt  &iter  & time  & time$_\text{eig}$
      &   fval-opt  &iter  & time  & time$_\text{eig}$
      &   fval-opt  &iter  & time  & time$_\text{eig}$
      &   fval-opt  &iter  & time  & time$_\text{eig}$  \\
  \midrule

  \texttt{BBM(AP)} & 3.2e-06 & 11.0 & 1.31e-02 & 8.11e-03
  & 7.4e-06 & 55.0 & 2.57e-02 & 8.22e-03
  & 8.6e-06 & 168.0 & 1.11e-01 & 6.77e-02
  & 8.5e-06 & 348.0 & 1.59e-01 & 7.49e-02 \\
  \texttt{BBM(SP)} & 8.4e-06 & 11.0 & 1.25e-02 & 8.11e-03
  & 9.9e-06 & 57.0 & 2.28e-02 & 8.22e-03
  & 7.5e-06 & 152.0 & 1.04e-01 & 6.77e-02
  & 2.8e-06 & 415.0 & 1.69e-01 & 7.49e-02 \\
  \texttt{APG(AP)}  & 2.2e-06 & 29.0 & 2.02e-02 & 8.11e-03
  & 1.0e-05 & 68.0 & 2.87e-02 & 8.22e-03
  & 9.6e-06 & 120.0 & 1.01e-01 & 6.77e-02
  & 9.9e-06 & 283.0 & 1.49e-01 & 7.49e-02 \\
  \texttt{APG(SP)} & 6.9e-06 & 12.0 & 1.35e-02 & 8.11e-03
  & 9.5e-06 & 70.0 & 2.78e-02 & 8.22e-03
  & 9.4e-06 & 135.0 & 1.06e-01 & 6.77e-02
  & 1.0e-05 & 324.0 & 1.59e-01 & 7.49e-02 \\
  \texttt{Krylov}  & 6.1e-06 & 9.0 & 4.05e-03 & 0
  & 8.2e-06 & 27.0 & 9.30e-03 & 0
  & 9.2e-06 & 68.0 & 2.17e-02 & 0
  & 1.0e-05 & 146.0 & 4.55e-02 & 0\\
  \midrule
  \midrule

  \multicolumn{17}{c}{$K=100,~n=100$}\\
  \midrule

  \multirow{2}{*}{Methods}
  &\multicolumn{4}{c|}{$\kappa=10$}
  &\multicolumn{4}{c|}{$\kappa=10^2$}
  &\multicolumn{4}{c|}{$\kappa=10^3$}
  &\multicolumn{4}{c}{$\kappa=10^4$} \\
  \cmidrule(r){2-5} \cmidrule(r){6-9} \cmidrule(r){10-13} \cmidrule(r){14-17}
      &  fval-opt  &iter  & time  & time$_\text{eig}$
      &   fval-opt  &iter  & time  & time$_\text{eig}$
      &   fval-opt  &iter  & time  & time$_\text{eig}$
      &   fval-opt  &iter  & time  & time$_\text{eig}$  \\
  \midrule

  \texttt{BBM(AP)} & 5.4e-06 & 13.6 & 1.07e-03 & 7.94e-04
  & 6.6e-06 & 55.2 & 1.60e-03 & 8.09e-04
  & 1.2e-05 & 104.1 & 2.35e-03 & 1.11e-03
  & 5.5e-04 & 276.6 & 4.23e-03 & 1.39e-03 \\
  \texttt{BBM(SP)} & 4.6e-06 & 15.2 & 1.06e-03 & 7.94e-04
  & 6.9e-06 & 56.3 & 1.53e-03 & 8.09e-04
  & 7.8e-06 & 101.8 & 2.20e-03 & 1.11e-03
  & 8.8e-04 & 286.4 & 4.21e-03 & 1.39e-03 \\
  \texttt{APG(AP)} & 5.4e-06 & 17.2 & 1.13e-03 & 7.94e-04
  & 8.1e-06 & 60.2 & 1.61e-03 & 8.09e-04
  & 8.7e-06 & 96.5 & 2.21e-03 & 1.11e-03
  & 1.4e-03 & 266.8 & 4.05e-03 & 1.39e-03 \\
  \texttt{APG(SP)} & 5.8e-06 & 18.0 & 1.11e-03 & 7.94e-04
  & 7.9e-06 & 60.1 & 1.60e-03 & 8.09e-04
  & 8.4e-06 & 90.3 & 2.12e-03 & 1.11e-03
  & 2.8e-03 & 295.4 & 4.38e-03 & 1.39e-03 \\
  \texttt{Krylov}  & 5.9e-06 & 9.0 & 7.99e-04 & 0
  & 7.2e-06 & 21.9 & 1.80e-03 & 0
  & 7.0e-06 & 33.1 & 2.32e-03 & 0
  & 6.9e-06 & 34.4 & 2.15e-03 & 0 \\
  \midrule
  \midrule

  \multicolumn{17}{c}{$K=100,~n=1000$}\\
  \midrule

  \multirow{2}{*}{Methods}
    &\multicolumn{4}{c|}{$\kappa=10$}
    &\multicolumn{4}{c|}{$\kappa=10^2$}
    &\multicolumn{4}{c|}{$\kappa=10^3$}
    &\multicolumn{4}{c}{$\kappa=10^4$} \\
  \cmidrule(r){2-5} \cmidrule(r){6-9} \cmidrule(r){10-13} \cmidrule(r){14-17}
    &  fval-opt  &iter  & time  & time$_\text{eig}$
    &   fval-opt  &iter  & time  & time$_\text{eig}$
    &   fval-opt  &iter  & time  & time$_\text{eig}$
    &   fval-opt  &iter  & time  & time$_\text{eig}$  \\
  \midrule

  \texttt{BBM(AP)}  & 5.2e-06 & 13.3 & 4.97e-03 & 3.20e-03
  & 6.9e-06 & 69.8 & 9.76e-03 & 3.61e-03
  & 8.8e-06 & 150.7 & 2.09e-02 & 8.10e-03
  & 1.3e-05 & 246.1 & 3.94e-02 & 1.71e-02 \\
  \texttt{BBM(SP)} & 4.5e-06 & 15.1 & 4.92e-03 & 3.20e-03
  & 6.7e-06 & 70.5 & 9.76e-03 & 3.61e-03
  & 8.0e-06 & 142.6 & 2.00e-02 & 8.10e-03
  & 1.1e-03 & 337.9 & 4.43e-02 & 1.71e-02 \\
  \texttt{APG(AP)}  & 5.0e-06 & 17.3 & 5.29e-03 & 3.20e-03
  & 8.6e-06 & 86.3 & 1.13e-02 & 3.61e-03
  & 9.2e-06 & 127.6 & 1.92e-02 & 8.10e-03
  & 2.4e-03 & 288.6 & 4.09e-02 & 1.71e-02 \\
  \texttt{APG(SP)}  & 7.8e-06 & 16.7 & 5.21e-03 & 3.20e-03
  & 8.6e-06 & 76.7 & 1.05e-02 & 3.61e-03
  & 9.5e-06 & 120.8 & 1.85e-02 & 8.10e-03
  & 3.2e-03 & 285.5 & 4.06e-02 & 1.71e-02 \\
  \texttt{Krylov} & 6.8e-06 & 9.0 & 2.52e-03 & 0
  & 7.9e-06 & 26.1 & 6.69e-03 & 0
  & 9.0e-06 & 60.6 & 1.50e-02 &0
  & 9.1e-06 & 87.0 & 2.08e-02 & 0 \\

  \bottomrule
  \end{tabular}
  \caption{Comparison between Krylov subspace methods and our methods for solving \eqref{opt:CRS} for different dimensions and sparsity levels in the \textbf{easy} case. Time unit: second.\label{tab:6}}
\end{sidewaystable}

\begin{sidewaystable}
  \tiny
  \centering
  \begin{tabular}{c|cccc|cccc|cccc|cccc}
  \toprule

  \multicolumn{17}{c}{$K=100,~n=5000$}\\
  \midrule

  \multirow{2}{*}{Methods}
  &\multicolumn{4}{c|}{$\kappa=10$}
  &\multicolumn{4}{c|}{$\kappa=10^2$}
  &\multicolumn{4}{c|}{$\kappa=10^3$}
  &\multicolumn{4}{c}{$\kappa=10^4$} \\
  \cmidrule(r){2-5} \cmidrule(r){6-9} \cmidrule(r){10-13} \cmidrule(r){14-17}
    &  fval-opt  &iter  & time  & time$_\text{eig}$
    &   fval-opt  &iter  & time  & time$_\text{eig}$
    &   fval-opt  &iter  & time  & time$_\text{eig}$
    &   fval-opt  &iter  & time  & time$_\text{eig}$  \\
  \midrule

  \texttt{BBM(AP)} & 4.9e-06 & 12.6 & 2.74e-02 & 1.56e-02
  & 7.1e-06 & 57.6 & 5.39e-02 & 1.75e-02
  & 6.7e-06 & 143.2 & 1.26e-01 & 5.29e-02
  & 2.6e-04 & 382.8 & 3.06e-01 & 1.21e-01 \\
  \texttt{BBM(SP)}  & 5.0e-06 & 12.7 & 2.68e-02 & 1.56e-02
  & 8.0e-06 & 58.2 & 4.95e-02 & 1.75e-02
  & 5.7e-06 & 159.6 & 1.32e-01 & 5.29e-02
  & 2.6e-04 & 390.9 & 3.07e-01 & 1.21e-01 \\
  \texttt{APG(AP)}  & 6.3e-06 & 14.1 & 2.77e-02 & 1.56e-02
  & 8.3e-06 & 69.0 & 5.64e-02 & 1.75e-02
  & 9.7e-06 & 115.0 & 1.13e-01 & 5.29e-02
  & 3.9e-04 & 344.1 & 2.90e-01 & 1.21e-01 \\
  \texttt{APG(SP)}  & 4.0e-06 & 18.5 & 3.04e-02 & 1.56e-02
  & 8.2e-06 & 57.8 & 5.07e-02 & 1.75e-02
  & 9.4e-06 & 121.4 & 1.16e-01 & 5.29e-02
  & 6.4e-04 & 305.8 & 2.71e-01 & 1.21e-01 \\
  \texttt{Krylov} & 6.4e-06 & 9.0 & 8.01e-03 & 0
  & 8.4e-06 & 26.5 & 2.37e-02 & 0
  & 9.4e-06 & 68.6 & 6.02e-02 & 0
  & 9.5e-06 & 134.9 & 1.16e-01 & 0 \\
  \midrule
  \midrule

  \multicolumn{17}{c}{$K=100,~n=10000$}\\
  \midrule

  \multirow{2}{*}{Methods}
  &\multicolumn{4}{c|}{$\kappa=10$}
  &\multicolumn{4}{c|}{$\kappa=10^2$}
  &\multicolumn{4}{c|}{$\kappa=10^3$}
  &\multicolumn{4}{c}{$\kappa=10^4$} \\
  \cmidrule(r){2-5} \cmidrule(r){6-9} \cmidrule(r){10-13} \cmidrule(r){14-17}
    &  fval-opt  &iter  & time  & time$_\text{eig}$
    &   fval-opt  &iter  & time  & time$_\text{eig}$
    &   fval-opt  &iter  & time  & time$_\text{eig}$
    &   fval-opt  &iter  & time  & time$_\text{eig}$  \\
  \midrule

  \texttt{BBM(AP)} & 6.4e-06 & 13.1 & 4.91e-02 & 2.70e-02
  & 7.2e-06 & 45.0 & 8.61e-02 & 2.83e-02
  & 8.1e-06 & 144.1 & 2.68e-01 & 1.02e-01
  & 1.0e-05 & 383.1 & 7.73e-01 & 3.38e-01 \\
  \texttt{BBM(SP)}  & 4.5e-06 & 13.0 & 4.86e-02 & 2.70e-02
  & 8.6e-06 & 47.8 & 8.87e-02 & 2.83e-02
  & 8.3e-06 & 153.0 & 2.78e-01 & 1.02e-01
  & 1.3e-04 & 361.7 & 7.46e-01 & 3.38e-01 \\
  \texttt{APG(AP)}  & 5.1e-06 & 17.1 & 5.50e-02 & 2.70e-02
  & 7.7e-06 & 46.1 & 8.92e-02 & 2.83e-02
  & 8.9e-06 & 130.9 & 2.57e-01 & 1.02e-01
  & 9.0e-04 & 381.9 & 7.80e-01 & 3.38e-01 \\
  \texttt{APG(SP)} & 5.6e-06 & 18.2 & 5.68e-02 & 2.70e-02
  & 7.0e-06 & 47.3 & 9.08e-02 & 2.83e-02
  & 9.5e-06 & 123.9 & 2.50e-01 & 1.02e-01
  & 9.4e-06 & 380.9 & 7.76e-01 & 3.38e-01 \\
  \texttt{Krylov}  & 6.5e-06 & 9.0 & 1.76e-02 & 0
  & 8.2e-06 & 26.1 & 5.12e-02 & 0
  & 9.5e-06 & 69.9 & 1.37e-01 & 0
  & 9.7e-06 & 147.3 & 2.91e-01 & 0 \\
  \midrule
  \midrule

  \multicolumn{17}{c}{$K=1000,~n=1000$}\\
  \midrule

  \multirow{2}{*}{Methods}
  &\multicolumn{4}{c|}{$\kappa=10$}
  &\multicolumn{4}{c|}{$\kappa=10^2$}
  &\multicolumn{4}{c|}{$\kappa=10^3$}
  &\multicolumn{4}{c}{$\kappa=10^4$} \\
  \cmidrule(r){2-5} \cmidrule(r){6-9} \cmidrule(r){10-13} \cmidrule(r){14-17}
    &  fval-opt  &iter  & time  & time$_\text{eig}$
    &   fval-opt  &iter  & time  & time$_\text{eig}$
    &   fval-opt  &iter  & time  & time$_\text{eig}$
    &   fval-opt  &iter  & time  & time$_\text{eig}$  \\
  \midrule

  \texttt{BBM(AP)}  & 6.0e-06 & 12.3 & 4.33e-02 & 2.38e-02
  & 6.3e-06 & 59.9 & 9.52e-02 & 2.54e-02
  & 7.8e-06 & 132.9 & 1.93e-01 & 5.60e-02
  & 7.0e-04 & 332.9 & 4.57e-01 & 1.16e-01 \\
  \texttt{BBM(SP)}  & 7.0e-06 & 13.9 & 4.46e-02 & 2.38e-02
  & 5.8e-06 & 55.3 & 8.77e-02 & 2.54e-02
  & 7.9e-06 & 127.0 & 1.87e-01 & 5.60e-02
  & 9.8e-04 & 283.1 & 4.05e-01 & 1.16e-01 \\
  \texttt{APG(AP)} & 5.0e-06 & 15.8 & 4.76e-02 & 2.38e-02
  & 7.5e-06 & 61.1 & 9.46e-02 & 2.54e-02
  & 8.9e-06 & 102.7 & 1.65e-01 & 5.60e-02
  & 1.2e-03 & 312.1 & 4.37e-01 & 1.16e-01 \\
  \texttt{APG(SP)}  & 5.6e-06 & 16.4 & 4.79e-02 & 2.38e-02
  & 7.7e-06 & 57.1 & 9.07e-02 & 2.54e-02
  & 9.5e-06 & 109.3 & 1.71e-01 & 5.60e-02
  & 1.8e-03 & 358.6 & 4.83e-01 & 1.16e-01 \\
  \texttt{Krylov} & 6.4e-06 & 9.0 & 1.69e-02 & 0
  & 8.0e-06 & 26.5 & 5.11e-02 & 0
  & 9.0e-06 & 59.9 & 1.13e-01 &0
  & 8.8e-06 & 81.4 & 1.58e-01 & 0 \\
  \midrule
  \midrule

  \multicolumn{17}{c}{$K=1000,~n=5000$}\\
  \midrule

  \multirow{2}{*}{Methods}
  &\multicolumn{4}{c|}{$\kappa=10$}
  &\multicolumn{4}{c|}{$\kappa=10^2$}
  &\multicolumn{4}{c|}{$\kappa=10^3$}
  &\multicolumn{4}{c}{$\kappa=10^4$} \\
  \cmidrule(r){2-5} \cmidrule(r){6-9} \cmidrule(r){10-13} \cmidrule(r){14-17}
    &  fval-opt  &iter  & time  & time$_\text{eig}$
    &   fval-opt  &iter  & time  & time$_\text{eig}$
    &   fval-opt  &iter  & time  & time$_\text{eig}$
    &   fval-opt  &iter  & time  & time$_\text{eig}$  \\
  \midrule

  \texttt{BBM(AP)}  & 4.9e-06 & 13.6 & 2.86e-01 & 1.61e-01
  & 6.0e-06 & 54.9 & 5.46e-01 & 1.66e-01
  & 8.1e-06 & 144.1 & 1.45e+00 & 5.33e-01
  & 9.7e-06 & 367.2 & 3.42e+00 & 1.15e+00 \\
  \texttt{BBM(SP)} & 5.8e-06 & 13.7 & 2.83e-01 & 1.61e-01
  & 7.6e-06 & 51.3 & 5.19e-01 & 1.66e-01
  & 7.3e-06 & 163.1 & 1.55e+00 & 5.33e-01
  & 7.8e-04 & 394.9 & 3.58e+00 & 1.15e+00 \\
  \texttt{APG(AP)}   & 5.7e-06 & 15.6 & 2.99e-01 & 1.61e-01
  & 8.8e-06 & 54.5 & 5.44e-01 & 1.66e-01
  & 9.2e-06 & 120.8 & 1.31e+00 & 5.33e-01
  & 9.9e-06 & 333.7 & 3.22e+00 & 1.15e+00 \\
  \texttt{APG(SP)}  & 5.4e-06 & 17.4 & 3.11e-01 & 1.61e-01
  & 8.1e-06 & 55.8 & 5.49e-01 & 1.66e-01
  & 9.4e-06 & 128.0 & 1.35e+00 & 5.33e-01
  & 5.3e-04 & 328.1 & 3.18e+00 & 1.15e+00 \\
  \texttt{Krylov} & 6.6e-06 & 9.0 & 1.02e-01 &0
  & 8.3e-06 & 26.1 & 2.98e-01 & 0
  & 9.4e-06 & 69.2 & 7.92e-01 & 0
  & 9.6e-06 & 131.7 & 1.52e+00 & 0\\

  \midrule
  \midrule


  \multicolumn{17}{c}{$K=1000,~n=10000$}\\
  \midrule

  \multirow{2}{*}{Methods}
  &\multicolumn{4}{c|}{$\kappa=10$}
  &\multicolumn{4}{c|}{$\kappa=10^2$}
  &\multicolumn{4}{c|}{$\kappa=10^3$}
  &\multicolumn{4}{c}{$\kappa=10^4$} \\
  \cmidrule(r){2-5} \cmidrule(r){6-9} \cmidrule(r){10-13} \cmidrule(r){14-17}
    &  fval-opt  &iter  & time  & time$_\text{eig}$
    &   fval-opt  &iter  & time  & time$_\text{eig}$
    &   fval-opt  &iter  & time  & time$_\text{eig}$
    &   fval-opt  &iter  & time  & time$_\text{eig}$  \\
  \midrule

  \texttt{BBM(AP)} &  4.9e-06 & 12.7 & 5.21e-01 & 2.91e-01
  & 7.5e-06 & 51.4 & 1.02e+00 & 3.01e-01
  & 7.7e-06 & 140.1 & 2.68e+00 & 8.24e-01
  & 1.5e-05 & 419.7 & 8.25e+00 & 2.95e+00 \\
  \texttt{BBM(SP)}  & 5.2e-06 & 13.5 & 5.27e-01 & 2.91e-01
  & 8.3e-06 & 46.1 & 9.51e-01 & 3.01e-01
  & 8.9e-06 & 131.8 & 2.56e+00 & 8.24e-01
  & 1.5e-05 & 428.3 & 8.33e+00 & 2.95e+00 \\
  \texttt{APG(AP)}  & 6.1e-06 & 14.1 & 5.50e-01 & 2.91e-01
  & 7.5e-06 & 51.6 & 1.04e+00 & 3.01e-01
  & 9.3e-06 & 107.8 & 2.26e+00 & 8.24e-01
  & 1.4e-05 & 364.3 & 7.58e+00 & 2.95e+00 \\
  \texttt{APG(SP)}  & 5.5e-06 & 15.3 & 5.65e-01 & 2.91e-01
  & 8.6e-06 & 50.1 & 1.01e+00 & 3.01e-01
  & 9.3e-06 & 105.3 & 2.23e+00 & 8.24e-01
  & 2.5e-04 & 364.5 & 7.58e+00 & 2.95e+00 \\
  \texttt{Krylov} & 6.6e-06 & 9.0 & 2.02e-01 & 0
  & 8.1e-06 & 26.5 & 6.12e-01 & 0
  & 9.3e-06 & 65.4 & 1.53e+00 &0
  & 9.6e-06 & 149.5 & 3.47e+00 &0 \\

  \bottomrule
  \end{tabular}
  \caption{Comparison between Krylov subspace methods and our methods for solving \eqref{opt:CRS} for different dimensions and sparsity levels in the \textbf{easy} case. Time unit: second.\label{tab:7}}
\end{sidewaystable}

\begin{sidewaystable}
  \tiny
  \centering
  \begin{tabular}{c|cccc|cccc|cccc|cccc}
  \toprule

  \multicolumn{17}{c}{$K=1000,~n=10000$}\\
  \midrule

  \multirow{2}{*}{Methods}
  &\multicolumn{4}{c|}{$ \textsf{gap} =10^{-1}$}
  &\multicolumn{4}{c|}{$ \textsf{gap} =10^{-2}$}
  &\multicolumn{4}{c|}{$ \textsf{gap} =10^{-3}$}
  &\multicolumn{4}{c}{$ \textsf{gap} =10^{-4}$} \\
  \cmidrule(r){2-5} \cmidrule(r){6-9} \cmidrule(r){10-13} \cmidrule(r){14-17}
    &  fval-opt  &iter  & time  & time$_\text{eig}$
    &   fval-opt  &iter  & time  & time$_\text{eig}$
    &   fval-opt  &iter  & time  & time$_\text{eig}$
    &   fval-opt  &iter  & time  & time$_\text{eig}$  \\
  \midrule

  \texttt{BBM(AP)}  & 8.2e-06 & 8.8 & 6.49e-01 & 4.47e-01
  & 7.7e-06 & 27.1 & 1.52e+00 & 1.09e+00
  & 8.6e-06 & 53.7 & 3.61e+00 & 2.83e+00
  & 8.8e-06 & 63.4 & 8.07e+00 & 7.18e+00 \\
  \texttt{BBM(SP)}  & 8.2e-06 & 8.8 & 6.41e-01 & 4.47e-01
  & 7.6e-06 & 27.4 & 1.52e+00 & 1.09e+00
  & 8.0e-06 & 54.7 & 3.59e+00 & 2.83e+00
  & 9.6e-06 & 62.2 & 8.04e+00 & 7.18e+00 \\
  \texttt{APG(AP)} & 5.6e-06 & 7.6 & 6.43e-01 & 4.47e-01
  & 7.5e-06 & 16.0 & 1.39e+00 & 1.09e+00
  & 9.2e-06 & 36.3 & 3.39e+00 & 2.83e+00
  & 9.9e-06 & 41.1 & 7.79e+00 & 7.18e+00 \\
  \texttt{APG(SP)} & 5.6e-06 & 7.6 & 6.41e-01 & 4.47e-01
  & 7.5e-06 & 16.0 & 1.39e+00 & 1.09e+00
  & 9.7e-06 & 35.8 & 3.38e+00 & 2.83e+00
  & 1.0e-05 & 39.1 & 7.78e+00 & 7.18e+00 \\
  \texttt{Krylov} & 8.8e-01 & 455.9 & 1.07e+01 &0
  & 1.9e+01 & 175.1 & 4.12e+00 & 0
  & 3.7e-03 & 442.5 & 1.04e+01 & 0
  & 1.1e-03 & 500.0 & 1.17e+01 & 0 \\

  \bottomrule
  \end{tabular}
  \caption{Comparison between Krylov subspace methods and our methods for solving \eqref{opt:CRS} for different dimensions and sparsity levels in the \textbf{hard} case. Time unit: second.\label{tab:9}}
\end{sidewaystable}
In this subsection, we compare the numerical performance between our methods and the Krylov subspace method \cite{cartis2011adaptive} using randomly generated instances.
The problem instances are generated  in the same manner as in \cite{carmon2018analysis}, except that we replace both the original diagonal matrix $A$ and vector $b$ by  $Q^TAQ$ and $Q^Tb$, respectively to make the problem more computationally involved and less trivial.
The matrix $Q$ is a random block diagonal matrix (with $n/K$ blocks) and each block is generated by the MATLAB command \texttt{orth(rand(K))} with $K$ being a positive integer. Note that the random matrices generated in this manner are of full rank almost surely. As pointed out in \cite{carmon2018analysis}, by construction, the optimal values are $-1$ for all cases.
Problems with different dimensions $n$ and different sparsity levels were tested.
The sparsity of matrix $A$ is then $K/n$, i.e., a proportion $K/n$ of the total entries are nonzero.
For fixed $K$ and $n$, problems with different condition numbers $\kappa$ and eigen-gaps \textsf{gap} (to be defined later) in the easy and hard cases were also tested, which are believed to strongly affect the hardness of problem \eqref{opt:CRS} and the Krylov subspace method \cite{carmon2018analysis}. In the easy case, we tested problems with the condition number $\kappa=\frac{\lambda_n+\lambda^*}{\lambda_1+\lambda^*}$, where $\lambda_n$ is the largest eigenvalue of $A$ and $\lambda^*$ is the optimal Lagrangian multiplier, which is an indicator for the hardness of the problem \cite{carmon2018analysis}.
In the hard case, we tested problems with different eigen-gap $\textsf{gap}=\lambda_{2}-\lambda_1$, where $\lambda_2$ is the second smallest eigenvalue of matrix $A$. All experiments were run on a Windows workshop with   16 Intel Xeon W-2145 cores (3.70GHz) and 64GB of RAM. 

The approximate eigenvalue in formulating the surrogate problem was computed by the MATLAB function \texttt{eigs}. We found empirically that setting the tolerance (an input argument of the MATLAB function \texttt{eigs}) to be $5/\kappa$ in the easy case and $10^{-6}$ in the hard case yields a reasonable trade-off between accuracy and efficiency.
Both \eqref{opt:BCP_p} and \eqref{opt:AP} were tested.
Besides APG, we have also applied BBM to solve the problems \eqref{opt:BCP_p} and \eqref{opt:AP}.
For APG, we used a restarting strategy, which is a common method for speeding up the algorithm \cite{o2015adaptive,ito2017unified}.
For BBM, we used a simple line search rule to guarantee the decrease of the objective function values.
As we know the optimal value is $-1$, we terminate our algorithm and  the Krylov subspace method\footnote{The authors are indebted to Coralia Cartis for her kind sharing of the MATLAB codes for the Krylov subspace method.} if the objective value is less than \texttt{-1+1e-6}.


Tables \ref{tab:6}--\ref{tab:9} show the  performance comparison of our methods and the Krylov subspace method.
In the tables,
 \texttt{BBM(AP)} denotes the method that solves problem~\eqref{opt:AP} by BBM; \texttt{BBM(SP)} denotes the method that solves problem~\eqref{opt:BCP_p} by BBM; \texttt{APG(AP)} denotes the method that solves problem~\eqref{opt:AP} by APG; \texttt{APG(SP)} denotes the method that solves problem~\eqref{opt:BCP_p} by APG; and \texttt{Krylov} denotes the Krylov subspace methods for directly solving problem~\eqref{opt:CRS}.
In the tables, \texttt{fval-opt}  denotes the objective value accuracy, which is the objective value returned by the algorithm minus the optimal value;
\texttt{iter} denotes the iteration number of each algorithm;
\texttt{time} denotes the total time of each algorithm;
\texttt{time$_\text{eig}$} denotes the time cost for approximately computing the minimum eigenvalue, which is 0 for the Krylov subspace method.

From Tables  \ref{tab:6} and \ref{tab:7} , we see that in the easy case, our methods achieved the prescribed accuracy when $\kappa<10^4$
and were a bit slower than  the Krylov subspace method.
All our four methods took more iterations and CPU time as the condition number $\kappa$ increases, as expected.
We also obtain that in our methods the eigenvalue computation took about 1/3 to 1/2 of the total CPU time and
the ratio of \texttt{time$_\text{eig}$} over \texttt{time}  becomes slightly smaller as the condition number increases.
From  Table  \ref{tab:9}, we see that in the hard case, our methods performed much better than the Krylov subspace method in terms of solution quality, iteration number and  CPU time. 
All our four methods took more iterations and CPU time as the eigen-gap $\kappa$ increases, as expected.
We also observe that the eigenvalue computation took more than 2/3 of total time and 
the ratio of time$_\text{eig}$ over time becomes larger if the eigen-gap decreases.
For the test problems with \texttt{gap}$=10^{-4}$, the Krylov subspace method attained the maximum time 500 seconds and failed to return a solution satisfying the stopping criteria, while
our methods sufficiently solved all the problems in less than 10 seconds on average.
As our methods always outperform the Krylov subspace method in the hard case, we do not report more results for the hard case.
In fact, the Krylov subspace method fails to find an approximate solution, while our methods always find a good approximate solution with an accuracy $10^{-6}$.
We also notice that, in both the easy and hard cases, APG are slightly better than BBM, especially for instances with a large condition number,
and each of APG and BBM has a similar performance on solving (AP) and (SP).
Comparing the ratio time$_\text{eig}$/time, we conclude that the two considered first-order methods performs on par in terms of solving the surrogate problems \eqref{opt:AP} and \eqref{opt:BCP_p}. For future research, we would like to develop more efficient methods for solving the surrogate problem.

\subsection{Numerical tests on CUTEst problems}
In this subsection, we compare the numerical performance of ARC algorithms (\cite{cartis2011adaptive}) implemented with different
subproblem solvers on unconstrained test problems of the CUTEst collections.

Towards that end, we describe a variant of ARC, Algorithm~\ref{alg:arc}, whose subproblem solver is based on our reformulation.
Denoting the function to minimize by $F$, in each iteration, we compute an approximate solution for the cubic regularization model function
\[\min_s m_k(s)=s^TB_ks+g_k^Ts+\frac{\sigma_k}{3}\|s\|^3,\]
where $B_k$ is an approximation of the Hessian $\nabla^2F(x_k)$, $g_k=\nabla F(x_k)$ and $\sigma_k$ is an adaptive parameter.
Suppose $\mathcal S$ is an arbitrary  solver for \eqref{opt:CRS}
and   $\mathcal A$ is an arbitrary   solver for the surrogate problem \eqref{opt:AP}.
In our algorithm, we call $\mathcal A$ if the following condition is met:
\begin{equation}
\label{eq:solvercon}
\|g_k\|\le\max\left(F(x_{k}),1\right)\cdot \epsilon_1\quad\text{ and }\quad\lambda_1(B_k)<-\epsilon_2,
\end{equation}
where $\epsilon_1$ and $\epsilon_2$ are some small positive real numbers and $\lambda_1(B_k)$ is the minimum eigenvalue of $B_k$;
and otherwise we call $\mathcal S$ to solve the model function directly. Condition~\eqref{eq:solvercon} is motivated by the facts that the Cauchy point is a good initial point when the norm of the gradient is large and that the subproblem solver $\mathcal A$ is designed for cases where $B_k$ at current iterate has at least one negative eigenvalue.
We use the Cauchy point \cite{cartis2011adaptive} as an initial point:
\[s_k^C=-\alpha_k^C\text{ and }\alpha_k^C=\argmin_{\alpha\in\R_+} m_k(-\alpha g_k).\]
The (approximate) solution $s_k$ to the model function returned by the solver $\mathcal S$ or $\mathcal A$ is accepted as the trial step if the model function value at $s_k$ is smaller than that at the Cauchy point $s_k^C$; otherwise the Cauchy point $s_k^C$ is used.
From \cite[Lemma 2.1]{cartis2011adaptive}, the above choice of the trial step guarantees that our variant of ARC (Algorithm \ref{alg:arc}) converges to a first-order stationary point (i.e., $\lim_{k\rightarrow\infty}\|g_k\|=0$)
under some mild conditions, e.g., $F$ is a continuously differentiable function, $\|g_{t_i}-g_{l_i}\|\rightarrow0$
whenever $\|x_{t_i}-x_{l_i} \|\rightarrow0$ for any subsequences $\{t_i\}$ and $\{l_i\}$ of natural numbers and the norm of the approximate Hessian $B_k $ is upper bounded by some positive constant for all $k$.

\begin{algorithm}[!ht]
\begin{algorithmic}[1]
\caption{ARC using reformulation \eqref{opt:AP}}
\label{alg:arc}
\Require $x_0,~\gamma_2\ge \gamma_1>1,~1>\eta_2\ge\eta_1>0,$ and $\sigma_0>0$, for $k=0,1,...$ until
convergence
\State compute the Cauchy point $s_k^C$
\If {condition \eqref{eq:solvercon} is satisfied}
\State compute a trial step  $\bar s_k$ using $\mathcal A$ with an initial point $(s_k^C,\|s_k^C\|)$
\Else
\State compute a trial step $\bar s_k$ using $\mathcal S$ with an initial point $s_k^C$
\EndIf
\State set
    \begin{equation*}
      s_{k}=
      \left\{
        \begin{array}{ll}
        \bar s_k&\text{if}~m_k(\bar s_k)\le m_k(s_k^C)  \\
        s_k^C& \text{otherwise}
        \end{array}
      \right.
    \end{equation*}
\State compute $f(x_k+s_k)$ and
    \[\rho_k=\frac{f(x_k)-f(x_k+s_k)}{-m_k(s_k)}\]
\State set
    \begin{equation*}
      x_{k+1}=
      \left\{
        \begin{array}{ll}
        x_k+s_k &\text{if}~\rho_k\ge\eta_1  \\
        x_k & \text{otherwise}
        \end{array}
      \right.
    \end{equation*}
\State  set
    \begin{equation*}
      \sigma_{k+1}\in
      \left\{
        \begin{array}{lll}
        \left(0,\sigma_k\right] &\text{if}~\rho_k>\eta_2 &(\text{very successful iteration})  \\
        \left[\sigma_k,\gamma_1\sigma_k\right] &\text{if}~\eta_1\le\rho_k\le\eta_2 &(\text{successful iteration})\\
        \left[\gamma_1\sigma_k, \gamma_2\sigma_k\right] & \text{otherwise} &(\text{unsuccessful iteration})
        \end{array}
      \right.
    \end{equation*}
\end{algorithmic}
\end{algorithm}

We implemented two different subproblem solvers $\mathcal A$ in Algorithm \ref{alg:arc}.
In our implementation, if condition \eqref{eq:solvercon} is not satisfied, we still solve the original cubic model \eqref{opt:CRS} by BBM so that the effect of solving problem~\eqref{opt:CRS} via our reformulation can be shown by comparing the overall time and solution quality.
For the cases where condition~\eqref{eq:solvercon} is satisfied, we implemented both APG and BBM to solve the surrogate problem \eqref{opt:AP}.
We call the former ARC-RAPG and the latter ARC-RBB.
We compare our algorithms against the ARC algorithm in \cite{cartis2011adaptive}, denoted by ARC-GLRT, where the subproblems are solved by the Krylov subspace method and another ARC algorithm with subproblems solved by BBM, denoted by ARC-BB, which is  a simplified version of the algorithm in \cite{bianconcini2015use}.

We tested medium-size ($n\in[500,1500]$) problems from the CUTEst collections as in \cite{bianconcini2015use}.
For condition~\eqref{eq:solvercon}, we set $\epsilon_1=10^{-2}$ and $\epsilon_2=10^{-4}$.
In our numerical tests, we found that condition \eqref{eq:solvercon} is never met for some problems,
and thus ARC-RAPG and ARC-RBB reduce to ARC-BB.
Therefore, we only report results on those instances where condition~\eqref{eq:solvercon} is satisfied in at lease one iteration. There are 20 such instances.
We implemented all the ARC algorithms in MATLAB 2017a.
All the experiments were run on a Macbook Pro  laptop. 
  The parameters in ARC are chosen as described in \cite{cartis2011adaptive}. We set $B_k=H_k:=\nabla^2 F(x_k)$ and the minimum eigenvalue of $H_k$ is  approximately computed by the MATLAB command \texttt{eigs} with tolerance \texttt{1e-4}.
  We terminate each ARC algorithm if either the iteration counter $k$ reaches 5000 or
  the stopping criteria \[\|g_k\| \le 10^{-5}\quad\text{and}\quad\lambda_1(H_k) \ge -10^{-3}\]
  are met.
  In our test, all algorithms were terminated before iteration counter reaches 5000.
 When using  APG to solve  \eqref{opt:AP} in ARC-RAPG, we again used the restarting strategy as in Section 4.1.
 For BBM to solve \eqref{opt:AP} or \eqref{opt:CRS} in ARC-BB, ARC-RAPG and ARC-RBB, we used the simple line search rule to guarantee the decrease of the objective function values as in Section 4.1.
  In all implementations, we terminate the subproblem solver in each inner iteration when the iteration number of the subproblem reaches 150, or the following stopping criterion is met:
  \[\|\nabla m_k(s_k)\|\le\min\{1,\|s_k\|\}\|g(x_k)\|.\]

We report the numerical results for ARC-RAPG, ARC-RBB, ARC-GLRT and ARC-BB in Table~\ref{tab:cutest}. The first column shows the name of the problem instance.
The number below the problem name represents its dimension.
The column $f^*$  shows the final objective function value.
The columns $n_i$, $n_{\text{prod}}$ , $n_f$ , $n_g$ and $n_{\text{eig}}$ show the iteration number, number of Hessian-vector products,
number of function evaluations, number of gradient evaluations and the number of eigenvalue computations.
The columns time, time$_\text{eig}$ and time$_\text{loop}$,
show in seconds the overall CPU time, eigenvalue computation time and difference between the last two, respectively.
From Table~\ref{tab:cutest}, we see that with our stopping criteria, the algorithms return the same final objective values on almost all the problem instances (except the problem CHAINWOO). The quantities $n_i$, $n_{\text{prod}}$, $n_f$, $n_g$ and $n_{\text{eig}}$ of the four algorithms are comparable.
From the table, we see that ARC-GLRT slightly outperforms the other three algorithms.
The table also shows that for some problems, our methods ARC-RAPG and ARC-RBB take much more CPU time than ARC-GLRT and ARC-BB.
We found that this is mainly because time$_\text{eig}$ is too large and dominates the total runtime in these problems.
This can be further divided into two situations: either there are many eigenvalue computations, or
the number of eigenvalue computations is small but each eigenvalue computation takes a large amount of time (in these cases, the MATLAB function \texttt{eigs} is difficult to converge, and we used \texttt{eig} instead to compute full eigenvalues).
  Excluding the time to compute the eigenvalues, the actual times of the four algorithms do not differ much, which is evidenced by the column time$_\text{loop}$.

To get more insights from the numerical tests, we also use performance profiles (\cite{dolan2002benchmarking}) to illustrate the experimental results in Figure \ref{fig:1}--\ref{fig:3}.
We note that, although ARC-GLRT has the best performance, the iteration numbers and the gradient evaluation numbers required by our algorithm are less than 2 times of those by ARC-GLRT on over 90\% of the tests, 
and the numbers of Hessian-vector products required by our algorithms are also less than 2 times of those by ARC-GLRT on about 80\% of the tests. We further note that ARC-RAPG, ARC-RBB and ARC-BB have the similar performance in terms of iteration numbers and gradient evaluations.
We also plot the performance profiles on test problems where algorithms require a CPU time of more than 1 second in Figure~\ref{fig:4}.
Compared with the previous results, the gap between our algorithms and ARC-GLRT has narrowed.

To see more advantages of our reformulation, we also investigate the numerical results for all the 10 realizations with different initial points for each problem.
In Table~\ref{tab:cutestfull}, we report the number that ARC-RAPG or ARC-RBB outperforms ARC-GLRT and ARC-BB out of the 10 realizations for each problem. We see that our methods outperform ARC-GLRT and ARC-BB frequently in terms of iteration number, the numbers of Hessian-vector products and gradient evaluations. This shows that our new reformulation may bring advantages in ARC algorithms.

{
\centering
\scriptsize
\begin{longtable}{c|c|ccccccccc}
\toprule
Problem & Method
  &  $n_i$  & $n_{\text{prod}}$ & $n_f$ & $n_g$
  & $n_{\text{eig}}$ & $f^*$  & time  & time$_\text{eig}$
  &  time$_\text{loop}$  \\

\midrule
&\texttt{ARC-GLRT}
&43.6
&804.8
&44.6
&35.7
&-
&2.45E+02
&0.270
&-
&0.270\\
BROYDN7D&\texttt{ARC-BB}
&42.6
&864.5
&43.6
&36.3
&1.0
&2.42E+02
&0.386
&0.051
&0.335\\
(1000)&\texttt{ARC-RAPG}
&42.5
&917.5
&43.5
&36.6
&15.7
&2.42E+02
&0.888
&0.558
&0.331\\
&\texttt{ARC-RBB}
&43.9
&941.2
&44.9
&36.7
&16.9
&2.42E+02
&0.929
&0.597
&0.332\\

\midrule
&\texttt{ARC-GLRT}
&40.3
&629.5
&41.3
&31.2
&-
&5.92E-14
&0.300
&-
&0.300\\
BRYBND&\texttt{ARC-BB}
&37.3
&644.5
&38.3
&30.4
&1.0
&3.24E-13
&0.308
&0.019
&0.289\\
(1000)&\texttt{ARC-RAPG}
&37.3
&670.4
&38.3
&30.4
&2.2
&3.10E-13
&0.327
&0.035
&0.292\\
&\texttt{ARC-RBB}
&37.5
&692.4
&38.5
&30.6
&2.4
&4.56E-13
&0.340
&0.037
&0.303\\

\midrule
&\texttt{ARC-GLRT}
&208.1
&5591.5
&209.1
&155.9
&-
&1.07E+03
&1.460
&-
&1.460\\
CHAINWOO&\texttt{ARC-BB}
&310.4
&12186.4
&311.4
&228.5
&1.0
&1.11E+03
&4.031
&0.324
&3.707\\
(1000)&\texttt{ARC-RAPG}
&330.2
&14027.7
&331.2
&231.8
&199.6
&1.11E+03
&26.609
&22.560
&4.049\\
&\texttt{ARC-RBB}
&311.2
&11232.0
&312.2
&230.6
&183.6
&1.13E+03
&24.114
&20.940
&3.174\\

\midrule
&\texttt{ARC-GLRT}
&24.9
&631.7
&25.9
&23.1
&-
&1.00E+00
&0.409
&-
&0.409\\
DIXMAANF&\texttt{ARC-BB}
&20.9
&534.4
&21.9
&20.7
&1.0
&1.00E+00
&0.541
&0.116
&0.426\\
(1500)&\texttt{ARC-RAPG}
&22.7
&628.1
&23.7
&20.9
&10.9
&1.00E+00
&1.399
&0.971
&0.428\\
&\texttt{ARC-RBB}
&21.7
&523.2
&22.7
&20.8
&9.8
&1.00E+00
&1.399
&0.971
&0.428\\

\midrule
&\texttt{ARC-GLRT}
&23.6
&523.4
&24.6
&22.6
&-
&1.00E+00
&0.351
&-
&0.351\\
DIXMAANG&\texttt{ARC-BB}
&24.0
&576.8
&25.0
&22.8
&1.0
&1.00E+00
&0.578
&0.117
&0.461\\
(1500)&\texttt{ARC-RAPG}
&25.8
&806.0
&26.8
&23.4
&11.4
&1.00E+00
&1.518
&1.007
&0.511\\
&\texttt{ARC-RBB}
&24.9
&578.2
&25.9
&23.0
&10.6
&1.00E+00
&1.388
&0.975
&0.413\\

\midrule
&\texttt{ARC-GLRT}
&26.5
&602.7
&27.5
&24.3
&-
&1.00E+00
&0.402
&-
&0.402\\
DIXMAANH&\texttt{ARC-BB}
&26.7
&720.8
&27.7
&24.7
&1.0
&1.00E+00
&0.660
&0.113
&0.547\\
(1500)&\texttt{ARC-RAPG}
&29.0
&1005.8
&30.0
&25.0
&12.7
&1.00E+00
&1.712
&1.086
&0.626\\
&\texttt{ARC-RBB}
&26.7
&668.7
&27.7
&24.7
&10.5
&1.00E+00
&1.416
&0.934
&0.481\\

\midrule
&\texttt{ARC-GLRT}
&46.9
&5031.1
&47.9
&39.6
&-
&1.00E+00
&2.458
&-
&2.458\\
DIXMAANJ&\texttt{ARC-BB}
&46.2
&3309.2
&47.2
&38.7
&1.0
&1.00E+00
&3.419
&1.545
&1.873\\
(1500)&\texttt{ARC-RAPG}
&52.8
&3736.0
&53.8
&40.0
&36.1
&1.00E+00
&34.097
&32.223
&1.873\\
&\texttt{ARC-RBB}
&48.3
&3104.6
&49.3
&41.9
&31.4
&1.00E+00
&33.251
&31.679
&1.572\\

\midrule
&\texttt{ARC-GLRT}
&62.2
&6228.8
&63.2
&50.6
&-
&1.00E+00
&3.058
&-
&3.058\\
DIXMAANK&\texttt{ARC-BB}
&70.3
&5865.4
&71.3
&55.9
&1.0
&1.00E+00
&4.773
&1.530
&3.242\\
(1500)&\texttt{ARC-RAPG}
&82.3
&6898.9
&83.3
&56.9
&62.2
&1.00E+00
&55.014
&51.676
&3.338\\
&\texttt{ARC-RBB}
&74.6
&5836.4
&75.6
&61.0
&54.3
&1.00E+00
&52.804
&49.958
&2.846\\

\midrule
&\texttt{ARC-GLRT}
&66.4
&6159.8
&67.4
&54.0
&-
&1.00E+00
&2.981
&-
&2.981\\
DIXNAANL&\texttt{ARC-BB}
&73.0
&5840.7
&74.0
&58.4
&1.0
&1.00E+00
&4.714
&1.505
&3.209\\
(1500)&\texttt{ARC-RAPG}
&83.0
&7475.0
&84.0
&59.8
&60.0
&1.00E+00
&52.685
&49.103
&3.582\\
&\texttt{ARC-RBB}
&76.1
&6458.4
&77.1
&63.3
&53.4
&1.00E+00
&55.178
&52.098
&3.080\\

\midrule
&\texttt{ARC-GLRT}
&1762.6
&51785.4
&1763.6
&1233.4
&-
&1.62E-08
&13.566
&-
&13.566\\
EXTROSNB&\texttt{ARC-BB}
&1368.5
&196671.1
&1369.5
&1139.3
&1.0
&2.97E-06
&49.540
&0.007
&49.533\\
(1000)&\texttt{ARC-RAPG}
&1386.6
&199155.1
&1387.6
&1116.3
&1277.4
&2.98E-06
&57.130
&8.713
&48.417\\
&\texttt{ARC-RBB}
&1393.9
&200360.0
&1394.9
&1118.7
&1284.9
&2.99E-06
&57.412
&8.750
&48.662\\

\midrule
&\texttt{ARC-GLRT}
&1790.8
&42046.0
&1791.8
&1215.2
&-
&7.97E-01
&10.844
&-
&10.844\\
FLETCHCR&\texttt{ARC-BB}
&1775.8
&49084.9
&1776.8
&1245.5
&1.0
&7.97E-01
&17.060
&0.012
&17.048\\
(1000)&\texttt{ARC-RAPG}
&1781.6
&49874.9
&1782.6
&1248.8
&567.9
&7.97E-01
&28.072
&9.042
&19.030\\
&\texttt{ARC-RBB}
&1830.3
&51176.3
&1831.3
&1272.7
&662.8
&7.97E-01
&30.574
&11.081
&19.493\\

\midrule
&\texttt{ARC-GLRT}
&35.8
&338.2
&36.8
&30.5
&-
&1.17E+05
&0.254
&-
&0.254\\
FREUROTH&\texttt{ARC-BB}
&36.4
&1378.7
&37.4
&31.2
&1.0
&1.17E+05
&0.487
&0.008
&0.479\\
(1000)&\texttt{ARC-RAPG}
&35.0
&1261.3
&36.0
&30.5
&23.0
&1.17E+05
&0.681
&0.199
&0.482\\
&\texttt{ARC-RBB}
&36.1
&1425.4
&37.1
&30.4
&23.8
&1.17E+05
&0.704
&0.198
&0.506\\

\midrule
&\texttt{ARC-GLRT}
&1676.4
&49839.0
&1677.4
&1027.0
&-
&2.47E-11
&12.462
&-
&12.462\\
GENHUMPS&\texttt{ARC-BB}
&1529.6
&41024.1
&1530.6
&926.6
&1.0
&1.55E-11
&14.545
&0.008
&14.537\\
(1000)&\texttt{ARC-RAPG}
&1529.5
&41090.9
&1530.5
&926.6
&9.1
&1.59E-11
&15.137
&0.132
&15.006\\
&\texttt{ARC-RBB}
&1529.5
&41081.9
&1530.5
&926.4
&9.1
&1.58E-11
&14.887
&0.134
&14.753\\

\midrule
&\texttt{ARC-GLRT}
&978.8
&20138.0
&979.8
&660.7
&-
&1.00E+00
&1.953
&-
&1.953\\
GENROSE&\texttt{ARC-BB}
&1097.5
&28322.9
&1098.5
&748.3
&1.0
&1.00E+00
&2.024
&0.004
&2.020\\
(500)&\texttt{ARC-RAPG}
&1097.5
&28850.6
&1098.5
&746.1
&221.3
&1.00E+00
&3.142
&0.804
&2.338\\
&\texttt{ARC-RBB}
&1093.1
&28496.0
&1094.1
&747.0
&218.1
&1.00E+00
&3.450
&0.935
&2.515\\

\midrule
&\texttt{ARC-GLRT}
&66.7
&8319.3
&67.7
&62.4
&-
&2.32E+03
&1.866
&-
&1.866\\
NONCVXU2&\texttt{ARC-BB}
&123.7
&8467.1
&124.7
&122.6
&1.0
&2.32E+03
&3.067
&0.697
&2.369\\
(1000)&\texttt{ARC-RAPG}
&116.1
&7497.0
&117.1
&116.0
&114.1
&2.32E+03
&67.010
&64.964
&2.046\\
&\texttt{ARC-RBB}
&128.5
&7936.8
&129.5
&122.6
&125.5
&2.32E+03
&71.118
&68.930
&2.188\\

\midrule
&\texttt{ARC-GLRT}
&264.2
&213656.6
&265.2
&258.6
&-
&2.32E+03
&43.103
&-
&43.103\\
NONCVXUN&\texttt{ARC-BB}
&1719.2
&240580.4
&1720.2
&1716.4
&1.0
&2.32E+03
&62.662
&0.753
&61.909\\
(1000)&\texttt{ARC-RAPG}
&1968.8
&277075.0
&1969.8
&1967.6
&1966.5
&2.32E+03
&1255.385
&1195.682
&59.704\\
&\texttt{ARC-RBB}
&2220.7
&312062.4
&2221.7
&2214.2
&2217.5
&2.32E+03
&1413.231
&1345.627
&67.604\\

\midrule
&\texttt{ARC-GLRT}
&41.7
&7781.1
&42.7
&33.0
&-
&4.55E-01
&0.774
&-
&0.774\\
OSCIPATH&\texttt{ARC-BB}
&72.0
&7779.9
&73.0
&66.1
&1.0
&4.55E-01
&0.718
&0.215
&0.502\\
(500)&\texttt{ARC-RAPG}
&99.7
&11934.9
&100.7
&75.7
&66.4
&4.55E-01
&15.259
&14.305
&0.954\\
&\texttt{ARC-RBB}
&75.8
&8345.5
&76.8
&64.4
&42.2
&4.55E-01
&10.241
&9.719
&0.522\\

\midrule
&\texttt{ARC-GLRT}
&15.7
&93.9
&16.7
&13.0
&-
&1.00E+01
&0.086
&-
&0.086\\
TOINTGSS&\texttt{ARC-BB}
&14.9
&270.8
&15.9
&12.1
&1.0
&1.00E+01
&0.128
&0.008
&0.121\\
(1000)&\texttt{ARC-RAPG}
&17.4
&613.9
&18.4
&13.7
&12.1
&1.00E+01
&0.319
&0.090
&0.229\\
&\texttt{ARC-RBB}
&17.6
&511.2
&18.6
&13.4
&12.3
&1.00E+01
&0.291
&0.096
&0.196\\

\midrule
&\texttt{ARC-GLRT}
&67.1
&306.9
&68.1
&55.2
&-
&9.77E-14
&0.354
&-
&0.354\\
TQUARTIC&\texttt{ARC-BB}
&75.6
&600.4
&76.6
&59.0
&1.0
&1.16E-11
&0.467
&0.010
&0.457\\
(1000)&\texttt{ARC-RAPG}
&76.0
&907.8
&77.0
&59.1
&7.5
&3.72E-10
&0.677
&0.107
&0.570\\
&\texttt{ARC-RBB}
&76.5
&730.8
&77.5
&59.8
&6.6
&1.74E-11
&0.570
&0.081
&0.489\\

\midrule
&\texttt{ARC-GLRT}
&294.8
&4663.0
&295.8
&217.5
&-
&2.16E-15
&1.625
&-
&1.625\\
WOODS&\texttt{ARC-BB}
&393.0
&9877.9
&394.0
&276.5
&1.0
&1.8E-14
&3.164
&0.009
&3.154\\
(1000)&\texttt{ARC-RAPG}
&393.7
&10102.2
&394.7
&275.9
&5.8
&6.9E-14
&3.559
&0.070
&3.489\\
&\texttt{ARC-RBB}
&392.8
&9847.9
&393.8
&276.5
&5.5
&1.99E-16
&3.535
&0.064
&3.471\\
\bottomrule
\caption{\small Results on the CUTEst problems}
\label{tab:cutest}
\end{longtable}
\begin{figure}[!h]
    \centering
    \includegraphics[width=12cm]{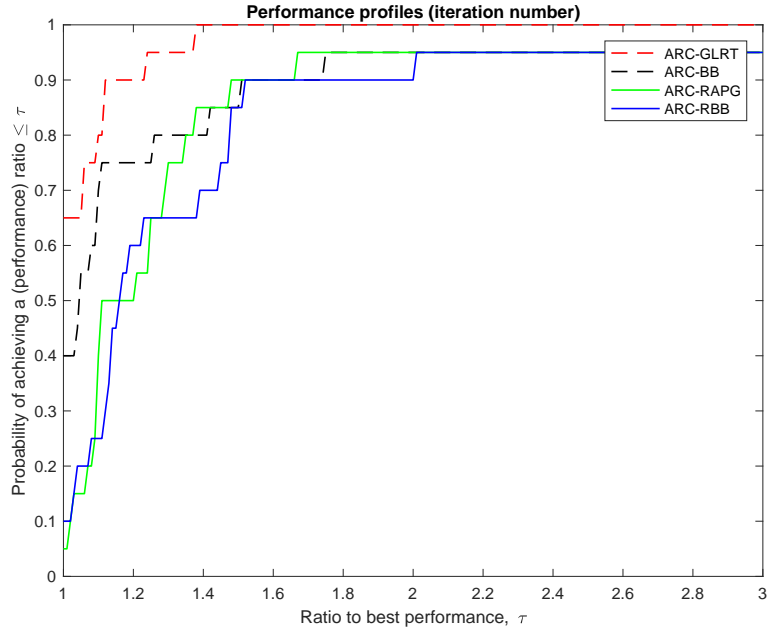}
    \caption{Performance profiles for iteration number for ARC-GLRT, ARC-BB, ARC-RAPG and ARC-RBB on the CUTEst problems}
    \label{fig:1}
  \end{figure}
  \begin{figure}[!h]
    \centering
    \includegraphics[width=12cm]{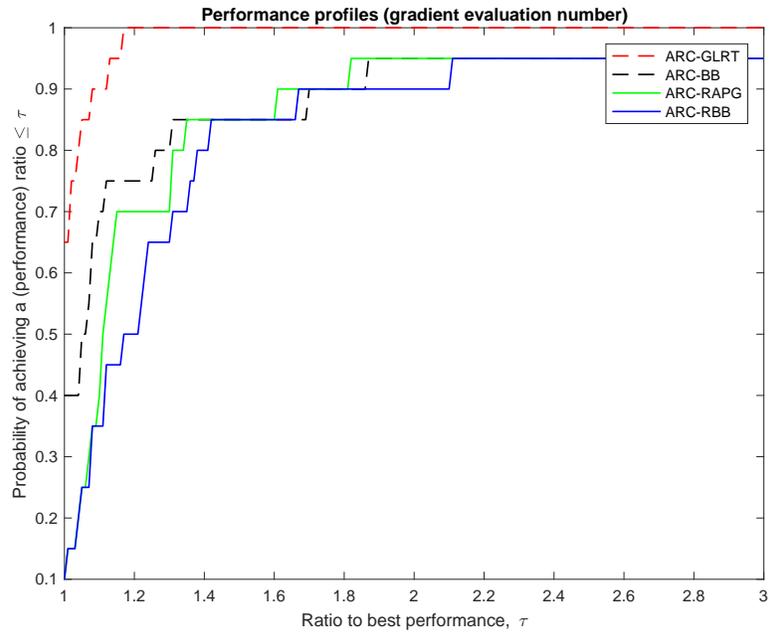}
    \caption{Performance profiles for gradient evaluations for ARC-GLRT, ARC-BB, ARC-RAPG and ARC-RBB on the CUTEst problems}
    \label{fig:2}
  \end{figure}
  \begin{figure}[!h]
    \centering
    \includegraphics[width=12cm]{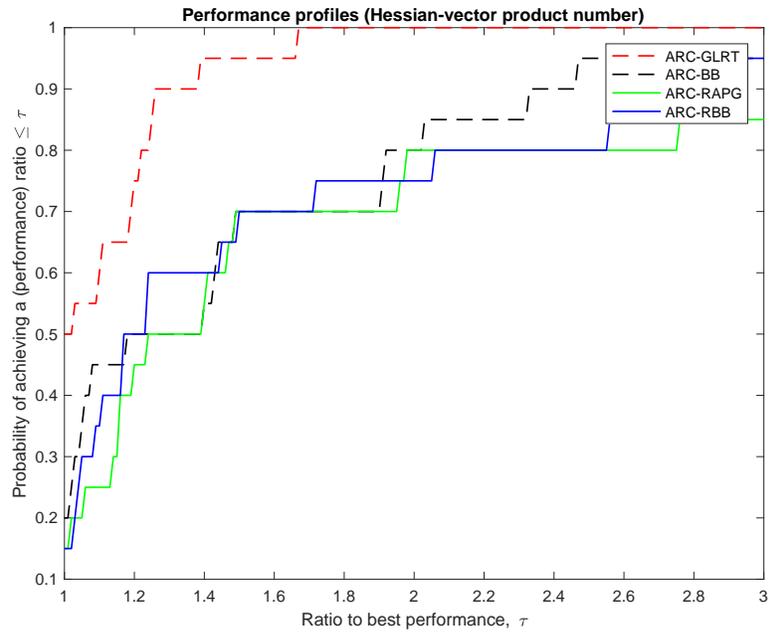}
    \caption{Performance profiles for Hessian-vector products for ARC-GLRT, ARC-BB, ARC-RAPG and ARC-RBB on the CUTEst problems}
    \label{fig:3}
  \end{figure}
  \begin{figure}[h]
    \centering
    \includegraphics[width=12cm]{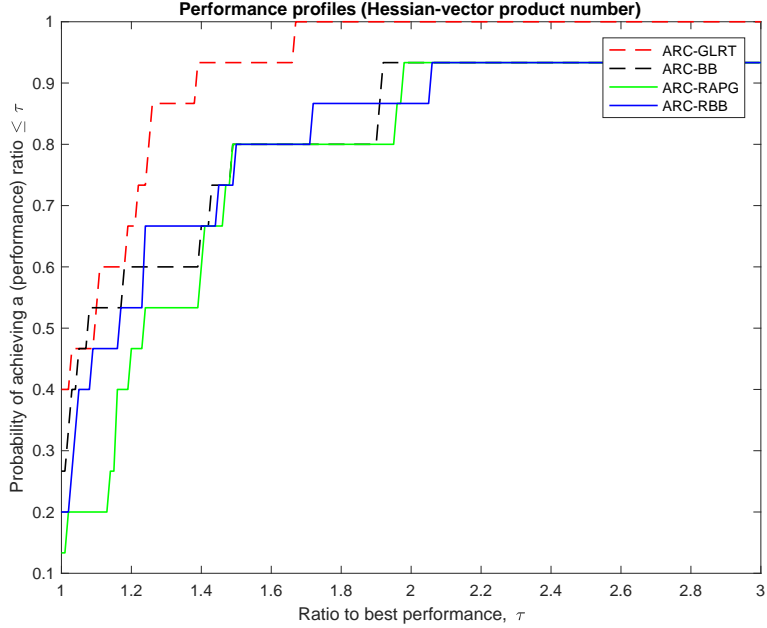}
    \caption{Performance profiles of Hessian-vector products for ARC-GLRT, ARC-BB, ARC-RAPG and ARC-RBB on the CUTEst problems where algorithms require a CPU time more than 1 second}
    \label{fig:4}
  \end{figure}
}

{\centering
\small
\begin{longtable}{c|c|cc|cc}

  \toprule
  \multirow{2}{*}{Problem}
  &\multirow{2}{*}{Index}
  &\multicolumn{2}{c|}{ARC-RAPG}
  &\multicolumn{2}{c}{ARC-RBB} \\
  \cmidrule(r){3-4} \cmidrule(r){5-6}
  &&  ARC-GLRT  & ARC-BB
  &  ARC-GLRT & ARC-BB \\

  \midrule
  \multirow{3}{*}{BROYDN7D}
  & $n_i$&6 &4 &6 &0\\
  & $n_{\text{prod}}$&4 &3 &4 &2\\
  & $n_g$ &5 &0 &4 &2\\

  \midrule
  \multirow{3}{*}{BRYBND}
  & $n_i$&6 &0 &6 &0\\
  & $n_{\text{prod}}$&4 &4 &4 &3\\
  & $n_g$ &6 &0 &6 &0\\

  \midrule
  \multirow{3}{*}{CHAINWOO}
  & $n_i$&0 &1 &0 &4\\
  & $n_{\text{prod}}$&0 &0 &0 &6\\
  & $n_g$ &0 &2 &0 &3\\

  \midrule
  \multirow{3}{*}{DIXMAANF}
  & $n_i$&8 &1 &9 &2\\
  & $n_{\text{prod}}$&5 &3 &7 &6\\
  & $n_g$ &9 &2 &9 &3\\

  \midrule
  \multirow{3}{*}{DIXMAANG}
  & $n_i$&3 &1 &3 &1\\
  & $n_{\text{prod}}$&3 &2 &3 &5\\
  & $n_g$ &3 &1 &3 &1\\

  \midrule
  \multirow{3}{*}{DIXMAANH}
  & $n_i$&5 &1 &5 &3\\
  & $n_{\text{prod}}$&3 &1 &4 &6\\
  & $n_g$ &4 &2 &5 &2\\

  \midrule
  \multirow{3}{*}{DIXMAANJ}
  & $n_i$&2 &2 &2 &4\\
  & $n_{\text{prod}}$&8 &3 &8 &5\\
  & $n_g$&4 &2 &2 &0\\

  \midrule
  \multirow{3}{*}{DIXMAANK}
  & $n_i$&1 &1 &0 &3\\
  & $n_{\text{prod}}$&3 &3 &6 &6\\
  & $n_g$&2 &2 &0 &0\\

  \midrule
  \multirow{3}{*}{DIXMAANL}
  & $n_i$&1 &2 &1 &4\\
  & $n_{\text{prod}}$&4 &3 &6 &4\\
  & $n_g$&2 &3 &0 &2\\

  \midrule
  \multirow{3}{*}{EXTROSNB}
  & $n_i$&10 &4 &10 &2\\
  & $n_{\text{prod}}$&0 &4 &0 &2\\
  & $n_g$&9 &7 &10 &7\\

  \midrule
  \multirow{3}{*}{FLETCHCR}
  & $n_i$&6 &5 &4 &1\\
  & $n_{\text{prod}}$&2 &3 &1 &2\\
  & $n_g$&3 &4 &2 &1\\

  \midrule
  \multirow{3}{*}{FREUROTH}
  & $n_i$&7 &7 &6 &4\\
  & $n_{\text{prod}}$&0 &6 &0 &6\\
  & $n_g$&6 &6 &4 &5\\

  \midrule
  \multirow{3}{*}{GENHUMPS}
  & $n_i$&10 &1 &10 &2\\
  & $n_{\text{prod}}$&10 &0 &10 &0\\
  & $n_g$&10 &1 &10 &2\\

  \midrule
  \multirow{3}{*}{GENROSE}
  & $n_i$&3 &6 &2 &5\\
  & $n_{\text{prod}}$&0 &2 &2 &5\\
  & $n_g$&2 &7 &1 &6\\

  \midrule
  \multirow{3}{*}{NONCVXU2}
  & $n_i$&0 &7 &0 &2\\
  & $n_{\text{prod}}$&6 &8 &7 &5\\
  & $n_g$ &0 &7 &0 &4\\

  \midrule
  \multirow{3}{*}{NONCVXUN}
  & $n_i$&0 &3 &0 &2\\
  & $n_{\text{prod}}$&2 &3 &2 &2\\
  & $n_g$&0 &3 &0 &2\\

  \midrule
  \multirow{3}{*}{OSCIPATH}
  & $n_i$&0 &1 &0 &3\\
  & $n_{\text{prod}}$&2 &1 &4 &3\\
  & $n_g$&0 &2 &0 &5\\

  \midrule
  \multirow{3}{*}{TOINTGSS}
  & $n_i$&4 &2 &3 &0\\
  & $n_{\text{prod}}$&0 &2 &0 &0\\
  & $n_g$&4 &3 &3 &0\\

  \midrule
  \multirow{3}{*}{TQUARTIC}
  & $n_i$ &2 &5 &2 &2\\
  & $n_{\text{prod}}$&0 &1 &0 &2\\
  & $n_g$&3 &5 &0 &2\\

  \midrule
  \multirow{3}{*}{WOODS}
  & $n_i$&0 &5 &0 &5\\
  & $n_{\text{prod}}$&0 &2 &0 &5\\
  & $n_g$&0 &4 &0 &3\\

  \bottomrule

  \caption{\small
  The number of times ARC-RAPG (or ARC-RBB) performs better than the other two algorithms in 10 realizations with different initial points on each CUTEst problem,
  considering the number of iterations, the number of Hessian-vector products and the number of gradient evaluations
  }
  \label{tab:cutestfull}
\end{longtable}
}

\section{Conclusion}
In this paper, we developed a novel approach for solving the problem~\eqref{opt:CRS}.
We first equivalently reformulate the problem~\eqref{opt:CRS} to a convex constrained optimization problem, where the feasible region admits an easy projection and the objective function is formed by using the minimum eigenvalue of the Hessian matrix. To circumvent the expensive cost due to the exact computation of the minimum eigenvalue, we then constructed a surrogate problem which is again a convex constrained optimization problem with a feasible region that admits an easy projection and can be solved by a variety of methods such as APG and BBM.
Furthermore, we proved that an $\epsilon$-approximate solution to \eqref{opt:CRS} can be obtained in at most  $\tilde O(\epsilon^{-1/2})$ matrix-vector multiplications if we use the Lanczos method for approximate eigenvalue computation and APG to approximately solve the surrogate problem.
Numerical results showed that our methods are comparable to the Krylov subspace method
in the easy case and significantly outperform the Krylov subspace method in the hard case.
We also implemented variants of ARC where the subproblem solver uses our approaches. The resulting ARC algorithms showed good numerical performance on the problem instances from the CUTEst datasets.
As future work, we plan  to investigate the complexity of cubic regularization or ARC variants with subproblem solver based on our reformulations for finding an approximate stationary point and local minimizer of smooth non-convex minimization problems.

\appendix
\section{Analysis for problem \eqref{opt:AP}}\label{app:b}
The purpose of this appendix is to show that when the approximate minimum eigenvalue $\theta$ of $A$ is close enough to the exact minimum eigenvalue $\lambda_1$ ($\lambda_1\le \theta<-\bar\lambda$ for some $\bar\lambda$ defined in the following paragraphs), problem \eqref{opt:AP}  can be used to construct an approximate solution to \eqref{opt:CRS}.
Define $\hat{B}:=\{(x,y): \norm{x}^2\leq y,~y\ge \hat{l}\,\}$. We claim that the problem \eqref{opt:AP} simplifies to
\begin{equation} \label{opt:AP0}\tag{AP$_0$}
\begin{array}{c@{\quad}l}
\displaystyle\min_{x\in \mathbb{R}^n, y\in \mathbb{R}} & \displaystyle \frac{1}{2} x^T(A-\theta I) x+b^Tx+\frac{\rho}{3}y^\frac{3}{2}+\frac{\theta}{2}y \\
\noalign{\smallskip}
\mbox{subject to} & \displaystyle \norm{x}^2\leq y,
\end{array}
\end{equation}
i.e., the constraint $y\ge l$ is redundant.
This is because when $y\le l$ in \eqref{opt:AP0}, $\frac{\rho}{3}y^\frac{3}{2}+\frac{\theta}{2}y$ is decreasing in $y$, and thus the optimal solution of \eqref{opt:AP0} must satisfy $y\ge l$.
We consider the following two cases.

\textbf{The hard case:} Recall the optimality condition \eqref{eq:optCRS} for \eqref{opt:CRS}.
Note that $\|(A+\lambda I)^{\dagger}b\|^2-\lambda^2/\rho^2\le 0$ and $\|(A+\lambda I)^{\dagger}b\|^2-\lambda^2/\rho^2$ is a decreasing function in $\lambda$, where $(\cdot)^\dagger$ denotes the \emph{Moore--Penrose pseudoinverse}, when $\lambda\ge-\lambda_1$ because we are  in the hard case \cite{cartis2011adaptive}.
First consider the case that  $\|(A-\lambda_{1} I)^{\dagger}b\|^2-\lambda_1^2/\rho^2< 0$. Let $\bar\lambda$ be the largest $\lambda\in[0,-\lambda_1)$ such that $\|(A+\lambda I)^{\dagger}b\|^2-\lambda^2/\rho^2=0$, if it exists. If such $\bar\lambda$ does not exist, we set $\bar\lambda=0$. Using the fact  $\|(A+\lambda I)^{\dagger}b\|^2-\lambda^2/\rho^2<0$ for $\lambda\ge-\lambda_1$, we have
\begin{equation}
\label{eq:hardbar}
\|(A+\lambda I)^{\dagger}b\|^2-\frac{\lambda^2}{\rho^2} < 0,  ~\forall \lambda>\bar\lambda.
\end{equation}
Suppose that $\lambda_1\le \theta<-\bar\lambda$ and that $(x^\theta, y^\theta)$ is an optimal solution of \eqref{opt:AP0}. Let $\mu$ be the Lagrange multiplier corresponding to the constraint $\norm{x^\theta}^2\le y^\theta$. Then, the KKT condition of \eqref{opt:AP0} implies  $Ax^\theta-\theta x^\theta+b+2\mu x^\theta=0$ and $\rho\sqrt {y^\theta}/2+\theta/2-\mu=0$.
Due to $-\theta +2\mu>\bar\lambda$, we have $\|x^\theta\|^2<y^\theta,\forall \mu\ge0$ from \eqref{eq:hardbar}.
Using the complementary slackness $\mu(\|x^\theta\|^2-y^\theta)=0$, we have that $\mu=0$.
Thus, every possible stationary point of \eqref{opt:AP0} can be written as $(x^\theta,y^\theta)=((A-\theta I)^{\dagger}b+tv$, $(-\theta/\rho)^2)$, where $t$ is a scalar satisfying $\|x^\theta+tv\|\le\sqrt {y^\theta}$. This in turn yields an approximate optimal solution $x^\theta+tv$ to \eqref{opt:CRS}, where one should note that different $t$ yields the same objective value. Next, we consider the remaining case that $\|(A+\lambda_{1} I)^{\dagger}b\|^2-\lambda_1^2/\rho^2=0$.
This case is similar to the easy case, where we can recover an optimal solution if $\theta\in[\lambda_1,-\bar\lambda)$, where $\bar\lambda$ is the largest $\lambda\in[0,-\lambda_1)$ such that  $\|(A+\lambda I)^{\dagger}b\|^2-\lambda^2/\rho^2=0$. (If such $\bar\lambda$ does not exist, we take $\bar\lambda=0$.) The analysis is similar to the easy case below and hence omitted here.

\textbf{The easy case:} Recall that in the easy case we have a unique optimal solution $x^*$ satisfying $\rho\|x^*\|>-\lambda_1$ and $x^*=(A+\rho\|x^*\|I)^{-1}b$; see, e.g., Theorem 3.1 in \cite{cartis2011adaptive}. Let $\bar\lambda$ be the largest $\lambda\in[0,-\lambda_1)$ such that  $h(\lambda):=\|(A+\lambda I)^{\dagger}b\|^2-\lambda^2/\rho^2=0$, if it exists. If such $\bar\lambda$ does not exist, we take $\bar\lambda=0$.  
Then for all $\lambda\in(\bar\lambda,-\lambda_1)$, $h(\lambda)>0$ as $\lim_{\lambda\rightarrow-\lambda_1} h(\lambda)=+\infty$. 
This, together with the definition of $\bar\lambda$ and the fact that $h(\lambda)$ is a decreasing function on $(-\lambda_1,+\infty)$, implies that there is only one point, denoting $\tilde\lambda$ in $(\bar\lambda,+\infty)$ satisfying $\|(A+\lambda I)^{\dagger}b\|^2-\lambda^2/\rho^2=0$. Let $\tilde x=(A+\tilde\lambda I)^{\dagger}b$. The optimality condition \eqref{eq:optCRS} implies that $\tilde x$ is the unique optimal solution of \eqref{opt:CRS}.
We again suppose that $\lambda_1\le \theta<-\bar\lambda$.
If $\theta=\lambda_1$, \eqref{opt:AP0} reduces to the exact reformulation \eqref{opt:CP}.
Next, we consider the case that $\lambda_1<\theta<-\bar\lambda$. Assuming that $\mu=0$, the inequality $\lambda_1<\theta<-\bar\lambda$ implies that $\|x^\theta\|=\|(A-\theta I)^\dagger b\|>\sqrt {y^\theta}=-\theta/\rho $, which violates the constraint $\norm{x}^2\le y$.
Hence, we must have $\mu>0$, and thus we always have  $\|x^\theta\|=\sqrt {y^\theta}$, i.e., $\|(A+(-\theta+\mu)I)^{\dagger}b\|=(-\theta+\mu)/\rho$. This implies $-\theta+\mu=\lambda^*$. That is, we recover the optimal solution if  $\theta<-\bar\lambda$.
\section*{Acknowledgments}
Rujun Jiang is supported by National Natural Science Foundation of China under Grant 11801087. Man-Chung Yue is supported by the Hong Kong Research Grants Council under the grant 25302420.

%
%

\bibliographystyle{abbrv}
\bibliography{ref}

\begin{thebibliography}{10}

\bibitem{agarwal2017finding}
N.~Agarwal, Z.~Allen-Zhu, B.~Bullins, E.~Hazan, and T.~Ma.
\newblock Finding approximate local minima faster than gradient descent.
\newblock In {\em Proceedings of the 49th Annual ACM SIGACT Symposium on Theory
  of Computing}, pages 1195--1199. ACM, 2017.

\bibitem{barzilai1988two}
J.~Barzilai and J.~M. Borwein.
\newblock Two-point step size gradient methods.
\newblock {\em IMA Journal of Numerical Analysis}, 8(1):141--148, 1988.

\bibitem{beck2009fast}
A.~Beck and M.~Teboulle.
\newblock A fast iterative shrinkage-thresholding algorithm for linear inverse
  problems.
\newblock {\em SIAM Journal on Imaging Sciences}, 2(1):183--202, 2009.

\bibitem{bianconcini2015use}
T.~Bianconcini, G.~Liuzzi, B.~Morini, and M.~Sciandrone.
\newblock On the use of iterative methods in cubic regularization for
  unconstrained optimization.
\newblock {\em Computational Optimization and Applications}, 60(1):35--57,
  2015.

\bibitem{birgin2019newton}
E.~Birgin and J.~Mart{\'\i}nez.
\newblock A {N}ewton-like method with mixed factorizations and cubic
  regularization for unconstrained minimization.
\newblock {\em Computational Optimization and Applications}, 73(3):707--753,
  2019.

\bibitem{carmon2019gradient}
Y.~Carmon and J.~Duchi.
\newblock Gradient descent finds the cubic-regularized nonconvex {N}ewton step.
\newblock {\em SIAM Journal on Optimization}, 29(3):2146--2178, 2019.

\bibitem{carmon2018analysis}
Y.~Carmon and J.~C. Duchi.
\newblock Analysis of {K}rylov subspace solutions of regularized non-convex
  quadratic problems.
\newblock In {\em Advances in Neural Information Processing Systems}, pages
  10705--10715, 2018.

\bibitem{carmon2018accelerated}
Y.~Carmon, J.~C. Duchi, O.~Hinder, and A.~Sidford.
\newblock Accelerated methods for nonconvex optimization.
\newblock {\em SIAM Journal on Optimization}, 28(2):1751--1772, 2018.

\bibitem{cartis2011adaptive}
C.~Cartis, N.~I. Gould, and P.~L. Toint.
\newblock Adaptive cubic regularisation methods for unconstrained optimization.
  {P}art {I}: motivation, convergence and numerical results.
\newblock {\em Mathematical Programming}, 127(2):245--295, 2011.

\bibitem{dolan2002benchmarking}
E.~D. Dolan and J.~J. Mor{\'e}.
\newblock Benchmarking optimization software with performance profiles.
\newblock {\em Mathematical programming}, 91(2):201--213, 2002.

\bibitem{flippo1996duality}
O.~E. Flippo and B.~Jansen.
\newblock Duality and sensitivity in nonconvex quadratic optimization over an
  ellipsoid.
\newblock {\em European Journal of Operational Research}, 94(1):167--178, 1996.

\bibitem{Ghanbari2018}
H.~Ghanbari and K.~Scheinberg.
\newblock Proximal quasi-{N}ewton methods for regularized convex optimization
  with linear and accelerated sublinear convergence rates.
\newblock {\em Computational Optimization and Applications}, 69(3):597--627,
  2018.

\bibitem{golub2012matrix}
G.~H. Golub and C.~F. Van~Loan.
\newblock {\em Matrix Computations}.
\newblock Johns Hopkins University Press, 4th edition, 2013.

\bibitem{gould2015cutest}
N.~I. Gould, D.~Orban, and P.~L. Toint.
\newblock Cutest: a constrained and unconstrained testing environment with safe
  threads for mathematical optimization.
\newblock {\em Computational Optimization and Applications}, 60(3):545--557,
  2015.

\bibitem{ho2017second}
N.~Ho-Nguyen and F.~Kilinc-Karzan.
\newblock A second-order cone based approach for solving the trust-region
  subproblem and its variants.
\newblock {\em SIAM Journal on Optimization}, 27(3):1485--1512, 2017.

\bibitem{ito2017unified}
N.~Ito, A.~Takeda, and K.-C. Toh.
\newblock A unified formulation and fast accelerated proximal gradient method
  for classification.
\newblock {\em The Journal of Machine Learning Research}, 18(1):510--558, 2017.

\bibitem{jiang2019novel}
R.~Jiang and D.~Li.
\newblock Novel reformulations and efficient algorithms for the generalized
  trust region subproblem.
\newblock {\em SIAM Journal on Optimization}, 29(2):1603--1633, 2019.

\bibitem{kuczynski1992estimating}
J.~Kuczy{\'n}ski and H.~Wo{\'z}niakowski.
\newblock Estimating the largest eigenvalue by the power and {L}anczos
  algorithms with a random start.
\newblock {\em SIAM Journal on Matrix Analysis and Applications},
  13(4):1094--1122, 1992.

\bibitem{nesterov2006cubic}
Y.~Nesterov and B.~T. Polyak.
\newblock Cubic regularization of {N}ewton method and its global performance.
\newblock {\em Mathematical Programming}, 108(1):177--205, 2006.

\bibitem{nesterov1983method}
Y.~E. Nesterov.
\newblock A method for solving the convex programming problem with convergence
  rate ${O}(1/k^2)$.
\newblock {\em Soviet Mathematics Doklady}, 27(2):372--376, 1983.

\bibitem{o2015adaptive}
B.~O'Donoghue and E.~Candes.
\newblock Adaptive restart for accelerated gradient schemes.
\newblock {\em Foundations of Computational Mathematics}, 15(3):715--732, 2015.

\bibitem{royer2018complexity}
C.~W. Royer and S.~J. Wright.
\newblock Complexity analysis of second-order line-search algorithms for smooth
  nonconvex optimization.
\newblock {\em SIAM Journal on Optimization}, 28(2):1448--1477, 2018.

\bibitem{tropp2010computational}
J.~A. Tropp and S.~J. Wright.
\newblock Computational methods for sparse solution of linear inverse problems.
\newblock {\em Proceedings of the IEEE}, 98(6):948--958, 2010.

\bibitem{wang2017linear}
J.~Wang and Y.~Xia.
\newblock A linear-time algorithm for the trust region subproblem based on
  hidden convexity.
\newblock {\em Optimization Letters}, 11(8):1639--1646, 2017.

\bibitem{yue2019quadratic}
M.-C. Yue, Z.~Zhou, and A.~Man-Cho~So.
\newblock On the quadratic convergence of the cubic regularization method under
  a local error bound condition.
\newblock {\em SIAM Journal on Optimization}, 29(1):904--932, 2019.

\bibitem{yue2019family}
M.-C. Yue, Z.~Zhou, and A.~M.-C. So.
\newblock A family of inexact {SQA} methods for non-smooth convex minimization
  with provable convergence guarantees based on the {L}uo--{T}seng error bound
  property.
\newblock {\em Mathematical Programming}, 174(1-2):327--358, 2019.

\end{thebibliography}
\end{document}